\documentclass{amsart}

\usepackage{latexsym,amsthm,amssymb}

\DeclareMathAlphabet{\mathpzc}{OT1}{pzc}{m}{it}

\usepackage[leqno]{amsmath}

\usepackage{mathrsfs}

\usepackage[all]{xy}

\usepackage{color}
\usepackage{colordvi}

\usepackage{verbatim}
\usepackage{ifthen}
\newboolean{hide.proofs}
\setboolean{hide.proofs}{false}
\newtheorem{theorem}{Theorem}[section]
\newtheorem{lemma}[theorem]{Lemma}
\newtheorem{definition}[theorem]{Definition}
\newtheorem{proposition}[theorem]{Proposition}

\newtheorem{remark}[theorem]{Remark}
\newtheorem{noname}[theorem]{}

\newtheorem{lemma-conjecture}[theorem]{Lemma--Conjecture}

\newtheorem{corollary}[theorem]{Corollary}

\newtheorem{example}[theorem]{Example}

\newtheorem{notation}[theorem]{Notation}

\newtheorem{question}[theorem]{Question}

\numberwithin{equation}{theorem}

\renewcommand{\mathcal}{\mathscr}

\newcommand{\Cal}{\mathcal}

\newcommand{\SN}{{\mathcal{N}}}
\newcommand{\SO}{{\mathcal{O}}}

\newcommand{\ST}{{\mathcal{T}}}

\newcommand{\SY}{{\mathcal{Y}}}

\newcommand{\PP}{\mathbf{P}}
\newcommand{\ZZ}{\mathbf{Z}}

\newcommand{\NN}{\mathbf{N}}

\title[Deformations of hyperelliptic 
polarized varieties]
{Deformations of hyperelliptic and generalized hyperelliptic 
polarized varieties}

\author{Purnaprajna Bangere}
\author{Francisco Javier Gallego}
\author{\\ Miguel Gonz\'alez}

\address{Department of Mathematics, University of Kansas}
\email{purna@math.ku.edu}
\address{Departamento de \'Algebra, 
Geometr\'ia y Topolog\'ia and 
Instituto de Matem\'atica Interdisciplinar, 
Universidad Complutense de Madrid}
\email{gallego@mat.ucm.es}
\address{Departamento de \'Algebra, Geometr\'ia y Topolog\'ia, Universidad Complutense de 
Madrid}
\email{miglezan@gmail.com}



\thanks{The first author was partially 
supported by GRF funding 2142083-099. 
The second and the third authors were partially 
supported by grant MTM2015-65968-P and by UCM research group 910772. 
The first author  
is also grateful for the hospitality 
of Algebra, Geometry and Topology Department of Universidad Complutense de Madrid 
and for the support of grant MTM2015-65968-P during his visit.}

\begin{document}

\begin{abstract}
In this article we study the deformations of hyperelliptic
polarized varieties $(X,L)$ of dimension $m$ and 
sectional genus $g$ such that the image $Y$ of the morphism 
$\varphi$
induced by $|L|$ is smooth. 
If $L^m < 2g-2$, it is known that, by 
adjunction and the Clifford's
theorem, any deformation of $(X,L)$ is hyperelliptic. 
Thus, we focus on when $L^m=2g-2$ or $L^m=2g$. 
We prove that, if $(X,L)$ is Fano-K3, then, except when 
$Y$ is a hyperquadric, all deformations of $(X,L)$ are again 
hyperelliptic (if $Y$ is a hyperquadric, the general deformation
of $\varphi$ is an embedding). This contrasts with 
the situation of hyperelliptic canonical curves and hyperelliptic
K3 surfaces. 
If 
$L^m=2g$, then we prove that, in most cases, 
a general deformation 
of $\varphi$ is a finite morphism of degree $1$. 
This provides interesting examples of degree $2$ morphisms that 
can be deformed to morphisms of degree $1$.
We extend our results to so-called generalized hyperelliptic
polarized Fano, Calabi-Yau and general type varieties.
The solutions to these questions 
are closely intertwined 
with the existence or non existence of double structures on 
the algebraic varieties $Y$.  
We address this matter as well.
\end{abstract}

\maketitle

\section*{Introduction}

\smallskip

In this article we study the deformations of 
finite, degree $2$ morphisms $\varphi$. 
Two things may happen: 
\begin{enumerate}
 \item[(I)] either there are deformations of $\varphi$ 
 which are of degree $1$; or 
 \item[(II)] all deformations of $\varphi$ 
are of degree
$2$.
\end{enumerate}

\smallskip

\noindent The paradigmatic situation in which (I) happens is
when $\varphi$ is the canonical morphism of a hyperelliptic curve
of genus $g$, $g >2$. 
Thus, from the point of view of deformations, moduli and 
related issues, it is natural to look for a generalization of the notion 
of hyperellipticity to higher dimensions. There will naturally be 
more than one analogue of this notion and  we study such 
analogues from the perspective of deformations 
and moduli in this article. 
We first deal with the notion of hyperellipticity that carries over from the case of 
algebraic curves, before we introduce  another, more general notion 
(see Definition~\ref{defi.generalized.hyperelliptic}).
A canonical curve comes with a natural
polarization, namely, we have a polarized 
variety $(C,K_C)$. Therefore, if we want to know in which situations 
(I) occurs, we look at 
the concept of \emph{polarized hyperelliptic varieties}, 
which were first introduced by Fujita in \cite{Fujita}. 
We reformulate \cite[Definition 1.1]{Fujita} in a way more 
suitable for us:

\begin{definition}
 {\rm Let $X$ be a smooth variety and let $L$ be a polarization 
 (i.e., an ample line bundle) on $X$. We say that the 
 polarized variety $(X,L)$ is hyperelliptic if  
 $L$ is base--point--free and the complete linear series 
 $|L|$ induces a morphism of degree $2$ onto its image, which is 
 a variety $Y$ embedded in projective space as a variety 
 of minimal degree.}
\end{definition}

\noindent It is well known (see for instance \cite{EH}) that a 
variety of minimal degree is either linear projective space, 
the Veronese surface in $\PP^5$, a smooth hyperquadric, a
smooth rational normal scroll or a cone over any of those. 
For instance, 
polarized hyperelliptic curves are pairs $(C,L)$, where $C$ is 
hyperelliptic and $L=lg^1_2$, $1 \leq l \leq g$; in this case, 
$|L|$ induces a morphism of degree $2$ onto a rational normal curve
of degree $l$.

\smallskip
\noindent Then, starting with a polarized variety $(X,L)$ of 
dimension $m$, 
a deformation of the morphism $\varphi$ from $X$ to $\PP^N$, 
induced 
by $|L|$ 
gives rise to a deformation
$(\mathcal X, \mathcal L)$
of $(X,L)$ with $h^0(L_t)=h^0(L)$. 
 Conversely, let $(X,L)$ be a polarized hyperelliptic variety,
 let $\varphi$ be the morphism  from $X$ to $\PP^N$ induced 
by $|L|$ 
 and 
let 
 $(\mathcal X, \mathcal L)$ a deformation of $(X,L)$. 
 Since ampleness and base-point-freeness are open conditions, 
if $h^0(L_t)=h^0(L)$ (e.g., this happens, after shrinking
the base of $(\mathcal X, \mathcal L)$ if necessary
 if $H^1(L)=0$,
see \cite[Chapitre III, \S 7]{EGA} or \cite[Chapter 0, \S 5]{MFK},
then, from $\mathcal L$ we obtain  a deformation 
 of $\varphi$. 
 If $(X,L)$ is hyperelliptic, then 
 \begin{equation*}
 L^m=2h^0(L)-2m,
 \end{equation*}
so
 $h^0(\mathcal L_t)=h^0(L)$ is a necessary condition 
 for a deformation $(\mathcal X_t, \mathcal L_t)$  of $(X,L)$
 to be hyperelliptic. 
Thus we may rephrase the question on deformations 
of hyperelliptic
morphisms as a question on deformations of 
polarized hyperelliptic varieties:

\noindent 
\begin{question}
 {\rm Given a polarized hyperelliptic variety $(X,L)$ and 
considering deformations $(\mathcal X, \mathcal L)$
of $(X,L)$ with $h^0(\mathcal L_t)=h^0(L)$, 
are all of them hyperelliptic or are 
there nonhyperelliptic deformations, i.e., 
such that the morphism induced by 
$|\mathcal L_t|$ is of degree $1$?}
\end{question}

\noindent
We will denote by $g$ the \emph{sectional genus} of a 
polarized variety $(X,L)$.
The  
behavior of the deformations of $(X,L)$ is partly governed 
by the relationship between the degree $L^m$ and 
$g$.  This  parallels the well-known behavior of deformations of 
polarized hyperelliptic curves, which is also partly governed 
by the relationship between the degree of $L$  and the genus
$g$ of the curve. 
Indeed, the behaviour of curves 
has these straight forward implications 
for 
the deformations of hyperelliptic polarized varieties. We mention them before we proceed 
to non trivial cases that we handle in this article. 

\begin{remark}\label{remark.deformation.varieties} 
{\rm 
Let 
 $(X,L)$ be a polarized variety of dimension
$m$, $m >1$.
 \begin{enumerate}
  \item If 
 $(X,L)$ is hyperelliptic and
 $L^m < 2g-2$, all deformations $(\mathcal X, \mathcal L)$ with 
 $h^0(\mathcal L_t)$ constant
are hyperelliptic (see e.g. \cite[(7.5)]{Fujita}). 
This follows essentially from  
looking at the restriction of the polarization to a
general curve section of $(X,L)$ and 
from Clifford's theorem. 
\item If $(X,L)$ is hyperelliptic, $L^m=2g-2$ and $g=2$, 
then all deformations of $(X,L)$ are hyperelliptic. This is because the morphism 
induced by $|L|$ is two-to-one onto $\PP^m$, so  $H^1(L)=0$
(by projection formula, Leray spectral sequence and the vanishing of intermediate
cohomology in $\PP^m$), so $h^0(\mathcal L_t)=h^0(L)$ by 
\cite[Chapitre III, \S 7]{EGA} (see also \cite[Chapter 0, \S 5]{MFK}).  
Since $X$ is regular (by projection formula, Leray spectral sequence and the vanishing of intermediate
cohomology in $\PP^m$), 
then the claim follows from looking at the restriction of $L$ 
to the general curve 
section, which has genus $2$.  
Alternatively, the claim follows inmediately from $h^0(\mathcal L_t)=h^0(L)$ 
and $\mathcal L_t^m=L^m=2$.
\item If $(X,L)$ is hyperelliptic, $L^m=2g$ and $g=1$, 
then all deformations of $(X,L)$ are hyperelliptic. 
This follows as in (2), using now Riemann-Roch 
for the first of the two arguments. 
\item If $L^m > 2g$ and $X$ is regular, by looking again at the restriction of $L$ 
to a general curve 
section and 
Riemann--Roch, 
we conclude that $(X,L)$ is nonhyperelliptic.
 \end{enumerate}}
\end{remark}

\smallskip
\noindent 
Therefore the cases in which we can expect the deformations 
of a hyperelliptic polarized variety $(X,L)$ to
be nonhyperelliptic are if $L^m=2g-2$ and $g>2$ or $L^m=2g$ and $g>1$.
A (higher dimensional) \emph{Fano--K3} variety $(X,L)$
(see Definition~\ref{defi.Fano-K3}) 
has a  general surface section which is 
a polarized K3 surface.
If $(X,L)$ is a
Fano-K3 variety, then $L^m=2g-2$. Conversely, 
if $L^m=2g-2$, $m \geq 3$ and $H^1(L)=0$, then 
$(X,L)$ is a Fano-K3 variety (see
\cite[Propositions 5.18 (3), 5.19 and (6.7)]{Fujita}).
It is 
well known that, unless $g=2$,  
the general deformation of a hyperelliptic K3 surface is 
nonhyperelliptic 
(this is because, if $g >2$, hyperelliptic K3 surfaces 
form a closed locus
in the moduli space of polarized K3 surfaces).
For $m > 2$, 
the study of the deformations of Fano--K3 varieties 
$(X,L)$ is partially addressed
by Fujita (see \cite[Remark 7.7, Corollary 7.14]{Fujita}), 
who finds 
that there is a hyperelliptic deformation of $(X,L)$. 
In Section~\ref{Section.Fano-K3}  we 
settle this question completely and our 
result contrasts with the results for 
surfaces and curves which are general sections of $(X,L)$ 
(those sections are, 
respectively, polarized hyperelliptic K3 surfaces and canonically polarized
hyperelliptic curves). Indeed, 
if the image $i(Y)$ of 
the morphism 
induced by $|L|$ is smooth, 
we show that any deformation of a Fano--K3 variety
$(X,L)$ is hyperelliptic except if $Y$ is embedded as  
 a  hyperquadric (see 
Theorems~\ref{thm.Fano-K3.deform.hyperelliptic} and 
\ref{thm.Fano-K3.deform.nonhyperelliptic}). On the contrary,  
if $g >2$ and $(X,L)$ is either a canonically
 polarized hyperelliptic curve or a polarized hyperelliptic K3 surface, then 
 the general 
 deformation of $(X,L)$ (with $h^0(\mathcal L_t)$ 
 constant throughout the deformation if $X$ is a curve)
 is nonhyperelliptic. This is because, 
if $g > 2$, then 
 being hyperelliptic is not a numerical condition,   
 but a closed condition on both the moduli of curves and the moduli
 of polarized K3 surfaces.

\smallskip
\noindent The interesting case $L^m=2g$ 
was not previously known. 
In Section~\ref{Section.2g} 
we settle the question 
completely when $Y$ is smooth. We prove that, except for one case, 
(see Theorems~\ref{thm.2g.deform.nonhyperelliptic}, 
\ref{thm.2g.deform.nonhyperelliptic2},
and
\ref{thm.2g.deform.hyperelliptic})
$(X,L)$ deforms to a nonhyperelliptic polarized variety, 
in clear contrast with 
Fano-K3 varieties. If $Y$ is as in (2e) of 
Proposition~\ref{prop.2g.classification} we  overcome an added difficulty:  
in this case (see 
Proposition~\ref{prop.2g.hom.not.0}), 
it is not clear whether  
$H^1(\mathcal N_\varphi)$  vanishes, and hence, whether $\varphi$ is
unobstructed or not. Still, we find a way 
to study the deformations of $\varphi$ (see the proof of  
Theorem~\ref{thm.2g.deform.nonhyperelliptic2}).

\smallskip
\noindent 
The outcome of all explained so far is that (I) seems to be more
rare than (II), as only may occur if $L^m=2g-2$ or $L^m=2g$. 
In addition, as dimension increases, (I) becomes even more rare
(see Theorem~\ref{thm.Fano-K3.deform.hyperelliptic}) and also, 
families of hyperelliptic varieties become fewer than 
one would expect
(see Propositions~\ref{prop.FanoK3.classification} and 
\ref{prop.2g.classification}).

\smallskip
\noindent 
The above results compel us to
investigate a deeper structural reason 
on why these phenomena happen.
In this direction, we introduce another, broader 
generalization of hyperellipticity: 
 
\begin{definition}\label{defi.generalized.hyperelliptic}
{\rm Let $(X,L)$ a smooth polarized variety with $L$ base--point--free. 
Let
the morphism from $X$ to $\PP^N$, induced by $|L|$ be of degree $2$ 
onto its image $i(Y)$ ($i$ is the embedding of $Y$ in $\PP^N$). 
Let the variety $Y$ be smooth and isomorphic to any of this: 
\begin{enumerate}
 \item projective space; 
 \item a hyperquadric; 
 \item a projective bundle over $\PP^1$. 
\end{enumerate}
The variety $Y$ is \emph{not necessarily} embedded by $i$ 
as a variety 
of minimal degree in $\PP^N$. 
Then 
we say that 
$(X,L)$ is a 
generalized hyperelliptic polarized variety.}
\end{definition}

\noindent In Section~\ref{Section.Fano-K3} and 
\ref{Section.generalized.hyperelliptic} we show that, 
in most cases,  
the deformations of 
generalized hyperelliptic
 Fano, Calabi--Yau and general type polarized varieties are 
 morphisms of degree $2$ (see Theorems~\ref{thm.Fano.deform}, 
 \ref{thm.deformation.CY} and \ref{thm.deformation.general.type}). 
In particular, the results on Calabi--Yau varieties 
are in sharp contrast to the results for their lower 
dimensional analogue, namely, K3 surfaces. This contrast can also be traced
 to the question of the existence of double structures. 
Indeed, in \cite{K3carpets} and 
\cite{BMR} it is shown 
the existence of locally Gorenstein multiplicity two structures (named,
in this case, carpets and, in general, ribbons; see e.g. \cite{BE} for definitions), 
nonsplit and 
with the same invariants as smooth 
K3 surfaces, supported on rational normal scrolls and, more generally, 
on general embeddings of Hirzeburch surfaces 
and other rational surfaces. Contrarily, 
we show in  Corollary~\ref{cor.CY.ribbons} that there are no nonsplit
ribbons
with  Calabi--Yau invariants supported on varieties $Y$ as in 
Definition~\ref{defi.generalized.hyperelliptic}.

\smallskip
\noindent Deformations of finite morphisms 
 of degree $2$ 
 are linked to the existence or non existence of  
 ribbons 
with right invariants.  Indeed, the relationship between the
 deformation of morphisms and 
these double 
 structures appears in \cite{Fong}, \cite{ropes}, \cite{Enriques}, 
 \cite{smoothing} or \cite{canonical.double.covers.surfaces},  
 to mention just a few references. 
 This relationship is at the very heart of 
the proofs of results like \cite[Theorem 1.4]{deformation} and 
\cite[Theorem 1.5]{smoothing}, that will be used later on
(for a deeper insight on all this, 
see \cite{Gonzalez}). Thus, in this article,  we do not only address
the problem of the existence 
 of Calabi--Yau ribbons but 
we also look at  the existence of ribbons with
other invariants. In Corollaries~\ref{cor.Fano.ribbons} and 
\ref{cor.general.type.ribbons} we prove 
(with one 
exception, see Example~\ref{example.smoothable.ribbon}), that 
Fano 
and general type nonsplit ribbons, 
supported on varieties $Y$ as in 
Definition~\ref{defi.generalized.hyperelliptic}, do not exist.

\section{General results on deformations of morphisms}

\noindent In this section we prove several general results 
on deformations of 
morphisms, that we will apply in the remaining of the article.
First we make clear what we will mean when we talk of
a deformation of a morphism: 

\begin{definition}
 {\rm 
 Let $X$ be an algebraic, projective variety 
 and let 
 \begin{equation*}
  \varphi: X \longrightarrow \PP^N 
 \end{equation*}
 be a  morphism.
By a deformation of $\varphi$ 
 we mean a  flat family of morphisms 
 \begin{equation*}
 \Phi: \mathcal X \longrightarrow \mathbf P^N_Z
 \end{equation*}
 over a smooth, irreducible algebraic 
 variety $Z$ (i.e., $\Phi$ is a $Z$--morphism 
 for which $\mathcal X \longrightarrow Z$ is proper, 
flat and surjective) with a distinguished point $0 \in Z$
  such that 
\begin{enumerate}
\item $\mathcal X$ is irreducible and reduced;
\item $\mathcal X_0=X$ and $\Phi_0=\varphi$.
\end{enumerate}

\smallskip
\noindent
Unless otherwise specified, 
when we say that certain property (*) is satisfied by a deformation 
$\Phi$, we will mean that
there exists an open neighborhood $U$ 
of $0$ 
such that, for all $z \in U$, property (*) is satisfied by the 
\emph{fibers}  
\begin{equation*}
 \Phi_z: \mathcal X_z \longrightarrow \PP^N
\end{equation*}
 of 
$\Phi$.

 \smallskip
\noindent
That certain property (*) is satisfied 
by a general deformation will mean that 
there exists an open set $\widehat U$ of the base of an
algebraic formally semiuniversal deformation $\Phi$ of $\varphi$
such that, for all $z \in \widehat U$, the morphism 
$\Phi_z$ satisfies (*).

\smallskip
\noindent
 We will use analogous definitions 
for a deformation 
of a variety $X$ or for a deformation of 
a polarized variety $(X,L)$.}
\end{definition}

\begin{noname}\label{setup}
{\bf  Notation and conventions:} 
{\rm Throughout this article, unless otherwise stated,
we will use the following notation and conventions:
\begin{enumerate}
\item  We will work over an algebraically closed field 
$\mathbf k$ of characteristic $0$.
\item  $X$ and $Y$ will denote smooth, irreducible, algebraic 
projective varieties.
\item $\omega_X$ and $\omega_Y$ denote respectively 
the canonical bundle of $X$ and $Y$; $K_X$ and $K_Y$ 
denote respectively 
the canonical divisor of $X$ and $Y$. We use the same notation for
all the other varieties that appear. 
\item  $i$ will denote a projective embedding $i: Y \hookrightarrow \mathbf P^N$. 
In this case, $\mathcal I$  will denote the ideal sheaf of $i(Y)$ 
in $\mathbf P^N$. We will often abridge $i^*\mathcal O_{\mathbf P^N}(1)$ as $\mathcal O_Y(1)$.
\item $\pi$ will denote a finite morphism $\pi: X \longrightarrow Y$ of degree $n = 2$; 
in this case, $\mathcal E$ will denote  the trace--zero module of $\pi$ 
($\mathcal E$ is a line bundle on $Y$) and $B$ will be the branch divisor of $\pi$ (then 
$\mathcal E^{-2}=\mathcal O_Y(B)$). 
\item $\varphi$ will denote a projective morphism $\varphi: X \longrightarrow \mathbf P^N$ 
such that $\varphi= i \circ \pi$.
\end{enumerate}}
\end{noname}

\noindent We introduce a homomorphism defined in~\cite[Proposition 3.7]{Gonzalez}:

\begin{proposition}\label{morphism.miguel}
Let $\mathcal N_\pi$ and  $\mathcal N_\varphi$ be 
respectively
the normal bundles of $\pi$ and $\varphi$ and let 
$\mathcal N_{i(Y),\mathbf P^N}$ be the normal bundle of 
$i(Y)$ in $\PP^N$. 
There exists a homomorphism
\begin{equation*}
 H^0(\mathcal N_\varphi) \overset{\Psi}\longrightarrow 
 \mathrm{Hom}(\pi^*(\mathcal I/\mathcal I^2), \mathcal O_X),
\end{equation*}
that appears when taking cohomology on 
the exact sequence 
\begin{equation}\label{eq.normal.bundles}
0 \longrightarrow \mathcal N_\pi \longrightarrow \mathcal N_\varphi 
\longrightarrow \pi^*\mathcal N_{i(Y),\mathbf P^N} \longrightarrow 0,
\end{equation}
that arises from
commutative 
diagram~\cite[(3.3.2)]{Gonzalez}. Since
\begin{equation*}
\mathrm{Hom}(\pi^*(\mathcal I/\mathcal I^2), \mathcal O_X)=
\mathrm{Hom}(\mathcal I/\mathcal I^2, \pi_*\mathcal O_X)=\mathrm{Hom}(\mathcal I/\mathcal I^2, \mathcal O_Y) 
\oplus \mathrm{Hom}(\mathcal I/\mathcal I^2, \mathcal E),
\end{equation*}
the homomorphism $\Psi$ has two components
\begin{eqnarray*}
H^0(\mathcal N_\varphi) & \overset{\Psi_1}  \longrightarrow & \mathrm{Hom}(\mathcal I/\mathcal I^2, \mathcal O_Y) \cr
H^0(\mathcal N_\varphi) & \overset{\Psi_2} 
\longrightarrow & \mathrm{Hom}(\mathcal I/\mathcal I^2, \mathcal E).
\end{eqnarray*}
Taking cohomology on \eqref{eq.normal.bundles}, 
the homomorphism $\Psi$ fits in this long exact sequence of cohomology:
 \begin{multline}
\label{eq.normal.bundles.cohomology}
 H^0(\mathcal N_\varphi) 
 \overset{\Psi}
 \longrightarrow
 \mathrm{Hom}(\mathcal I/\mathcal I^2,
\mathcal O_Y)\oplus
 \mathrm{Hom}(\mathcal I/\mathcal I^2,
\mathcal E)  \overset{\epsilon}\longrightarrow \\
H^1(\mathcal N_\pi)  
\overset{\eta}\longrightarrow  
H^1(\mathcal N_\varphi) 
\longrightarrow
 \mathrm{Ext}^1(\mathcal I/\mathcal I^2,
\mathcal O_Y)\oplus
 \mathrm{Ext}^1(\mathcal I/\mathcal I^2,
\mathcal E). 
\end{multline}
\end{proposition}

\noindent We now prove a useful generalization of 
\cite[Theorem 2.6]{deformation}. 
The proof given here also settles 
an unclear point in the proof given in
\cite{deformation}.

\begin{theorem}\label{Psi2=0.3}
With the notation of \ref{setup},  
assume furthermore that
\begin{enumerate}
\item the line bundle $\mathcal E$ 
can be lifted to a line bundle on any 
infinitesimal deformation of $Y$ (or, equivalently, the map $H^1(\mathcal T_{Y}) \to H^2(\mathcal O_{Y})$ induced by the cohomology 
class $c_1(\mathcal E) \in H^1(\Omega_{Y})$ 
via cup product and duality is zero; $\Omega_Y$
is the cotangent bundle
of $Y$, see \cite[Theorem 3.3.11 (iii)]{Sernesi}); 
\item $H^1(\mathcal E^{-2})=0$;
\item $H^1(\mathcal O_Y)=0$; 
\item the variety $Y$ is unobstructed in $\mathbf P^N$ 
(i.e., the base of the universal deformation space of $Y$ in $\mathbf P^N$ 
is smooth); 
and 
\item $\Psi_2 = 0$.
\end{enumerate}
Then
\begin{enumerate}
\item[(a)] $\varphi$ are unobstructed, and
\item[(b)] there is an
algebraic formally semiuniversal deformation for $\varphi$ 
which is a (finite) morphism of degree $2$ over a deformation 
of $i(Y)$ in $\PP^N$; i.e., 
any deformation of $\varphi$ is a (finite)  
morphism of degree $2$ over a deformation 
of $i(Y)$ in $\PP^N$. 
\end{enumerate}
\end{theorem}

\begin{proof}
	Let $\mathrm{Def}_{(Y,B)}$ be the functor of 
	deformations of the pair 
	$(Y \subset \PP^N, B \in |\mathcal E^{-2} |)$.

	\smallskip
	\noindent
	Since $H^1(\mathcal O_Y)=0$, for any embedded infinitesimal 
	deformations $Y \subset \bar Y \subset \widetilde{Y}$ we see that 
	$\mathrm{Pic}(\widetilde Y) \hookrightarrow \mathrm{Pic}(\bar Y) 
	\hookrightarrow \mathrm{Pic}(Y) $. Then, from $(1)$ we see that 
	for any $Y \subset \bar Y$ there is a unique lifting 
	$(\bar Y, \bar{\mathcal E})$ of $(Y, \mathcal E)$. Then, 
	by uniqueness, for the liftings $(\widetilde Y, 
	\widetilde{\mathcal E})$ and  $(\bar Y, \bar{\mathcal E})$ 
	of $(Y, \mathcal E)$ we have $\widetilde{\mathcal E}_{|\bar Y}= 
	\bar{\mathcal E}$. Thus we see that the local Hilbert functor of 
	deformations of $B$ in $Y$ is 
	a subfunctor of $\mathrm{Def}_{(Y,B)}$, whose tangent space is

	\smallskip
	\noindent
	Let $H_Y^{ \mathbf P^N}$ denote the local Hilbert functor of $Y$ in $\mathbf P^N$. Then for any small extension $\widetilde A \twoheadrightarrow A$ and elements $\widetilde Y \in H_Y^{ \mathbf P^N}(\widetilde A), (\bar Y, \bar B) \in \mathrm{Def}_{(Y, B)}(A)$ such that $\widetilde Y \times \mathrm{Spec} \, A = \bar Y$, we have an exact sequence
	\begin{equation*}
	0 \to \mathcal E^{-2} \to \widetilde{\mathcal E}^{-2} \to \bar{\mathcal E}^{-2} \to 0
	\end{equation*}  
	and then an exact sequence
	\begin{equation*}
	H^0(\widetilde{\mathcal E}^{-2} ) \to H^0(\bar{\mathcal E}^{-2} ) \to  H^1(\mathcal E^{-2} )=0.
	\end{equation*}
	Therefore $\bar B$ can be lifted to $\widetilde B$ such that $(\widetilde Y, \widetilde B) \in \mathrm{Def}_{(Y, B)}(\widetilde A)$.
	
	\noindent
	Assuming $Y \subset \mathbf P^N$ is unobstructed, the map $H_Y^{ \mathbf P^N}(\widetilde A) \to H_Y^{ \mathbf P^N}(\bar A)  $ is surjective. Therefore the argument above shows that the map
	\begin{equation*}
	\mathrm{Def}_{(Y,B)}(\widetilde A) \to \mathrm{Def}_{(Y, B)}(\bar A)  
	\end{equation*}
	is surjective, which shows that the functor $\mathrm{Def}_{(Y, B)}$ and the forgetful map $\mathrm{Def}_{(Y, B)} \to H_Y^{ \mathbf P^N}$ are smooth.
	Therefore we have an exact sequence on tangent spaces
	\begin{equation}\label{tangent.D}
	0 \to H^0(\mathcal O_B(B))  \to \mathrm{Def}_{(Y, B)}(k[\epsilon]) \to H^0(\mathcal N_{Y, \mathbf P^N}) \to 0.
	\end{equation}
	
	\smallskip
	\noindent
	One can show that $\mathrm{Def}_{(Y, B)}$ has a 
	semiuniversal formal element by checking Schlessinger's 
	conditions (see e.g. \cite[Theorem 2.3.2]{Sernesi}).

	\smallskip
	\noindent
	There is a map
	\begin{equation}\label{map.F}
	\mathrm{Def}_{(Y,B)} \xrightarrow{F} \mathrm{Def}_{\varphi}
	\end{equation}
	defined as follows. Let $(\bar Y, \bar B)$ be an element in $\mathrm{Def}_{(Y,B)}(A)$, where $\bar B=(\bar r)_0$, with $\bar r \in H^0(\bar{\mathcal E}^{-2})$ lifting $r \in H^0(\mathcal E^{-2})$ such that $B=(r)_0$.  Then, on the total space of 
	$\bar{\mathcal E}^{-1}$, there is a tautological section $\bar t$ lifting the tautological section on the total space of $\mathcal E^{-1}$. Then $(\widetilde{X}, \widetilde{\varphi}) \in \mathrm{Def}_{\varphi}(A)$ is given by
	\begin{equation*}\label{}
	\xymatrix@C-15pt{
		(\bar t^2- \bar r)_{0}\ar@{=}[r] & \widetilde X \ar[d]^{\widetilde \pi}\ar@{^{(}->}[rrrr]& && &\bar{\mathcal E}^{-1}
		\ar[lllld]\\
		& \bar Y&&&&,}
	\end{equation*}
	and
	\begin{equation*}
	\widetilde \varphi : \widetilde X \xrightarrow{\widetilde \pi} \bar Y \hookrightarrow \mathbf P^N_A.
	\end{equation*}
	Then $\widetilde X$ is flat over $A$ and $(\widetilde X, \widetilde \varphi)$ is a lifting of $(X, \varphi)$.

	\smallskip
	\noindent
	Recall that $\mathrm{Def}_{\varphi}(k[\epsilon])=
	H^0(\mathcal N_{\varphi})$ and 
	the long exact sequence of 
	cohomology~\eqref{eq.normal.bundles.cohomology}.
	Since the restriction of $\pi$ becomes an isomorphism between the ramification divisor and the branch locus $B$, this isomorphism identifies $\mathcal O_B(B)=\mathcal N_{\pi}$.
	
	\smallskip
	\noindent
	Let $(\bar{Y}, \bar{B}) \in \mathrm{Def}_{(Y,B)}(k[\epsilon])$ be a first order deformation of $(Y \subset \mathbf P^N, B \in | \mathcal E^{-2} |)$ and $(\widetilde{X}, \widetilde{\varphi}) \in H^0(\mathcal N_{\varphi})$ be the first order deformation of $(X, \varphi)$ associated to $(\bar{Y}, \bar{B})$ by $\mathrm{d}F$. From the construction we made for $F$, we see that $\widetilde{X}\xrightarrow{\widetilde{\varphi}} \mathbf P^N_{k[\epsilon]}$ factors through $\bar{Y} \hookrightarrow \mathbf P^N_{k[\epsilon]} $. Therefore $\mathrm{im}\,\widetilde \varphi =\bar{Y}$. Then, using \cite[Theorem 3.8 (2) and Propositions 3.11, 3.12]{Gonzalez}, we see that there is a commutative diagram
	\begin{equation}\label{smart.diagram}
	\xymatrix@C-10pt{
		& 0 \ar[d] & 0 \ar[d] \\
		0 \ar[r] & H^0(\mathcal O_{B}(B)) \ar[d] \ar^\simeq[r] & H^0(\mathcal N_{\pi}) \ar[d] \\
		& \mathrm{Def}_{(Y,B)}(k[\epsilon]) \ar[d]_{\mathrm{d}\nu} \ar[r]^{\mathrm{d}F} & H^0(\mathcal N_{\varphi}) \ar[d]^{\Psi_1 \oplus \Psi_2} \\
		0 \ar[r] & H^0(\SN_{Y/\mathbf P^N}) \ar[d] \ar[r] & H^0(\mathcal N_{Y/\mathbf P^N})\oplus H^0(\mathcal N_{Y/\mathbf P^N} \otimes \mathcal E) \\
		& 0&  .
	}
	\end{equation}
	We see, from diagram \eqref{smart.diagram}, that if $\Psi_2=0$, then $\mathrm{d}F$ is an isomorphism.
	
	\noindent
	Since $\mathrm{Def}_{(Y, B)}$ and $\mathrm{Def}_{\varphi}$ have a semiuniversal formal element, $\mathrm{Def}_{(Y, B)}$ is smooth and $\mathrm{d}F$ is an isomorphism it follows that $\mathrm{Def}_{\varphi}$ and $F$ are smooth. This shows that $\varphi$ is unobstructed.
	
	\smallskip
	\noindent
	Now we prove that any deformation of $\varphi$ is of degree $2$.
	
	\noindent
	Let $p: \mathcal Y \hookrightarrow \mathbf P_{U}^N \to 
	(U, u_0)$ be 
	an algebraic formally universal embedded deformation of $Y$.
	We can, by assumption, take $U$ and the total family $\mathcal Y$ smooth.

	\smallskip
	\noindent
	Since, from $(1)$, the line bundle $\mathcal E$ can be lifted to any infinitesimal deformation of $Y$, it follows that, after an etale base change centered at $u_0 \in U$, there exists a lifting of $\mathcal E$ to a line bundle $\mathbb L$ on $\mathcal Y$. Since, by the election of $\mathcal Y$, for any infinitesimal lifting $(\bar Y, \bar{\mathcal E})$ of $(Y, \mathcal E)$ there is a base change $\mathrm{Spec} \, A \to U$ such that $\mathcal Y \times \mathrm{Spec} \, A= \bar Y$ then, by the uniqueness of $ \bar{\mathcal E}$,  we see that ${\mathbb L}_{\mid {\bar Y}} = \bar{\mathcal E}$.

	\smallskip
	\noindent
	By the hypothesis $h^1(\mathcal E^{-2})=0$, we can assume $h^1({\mathbb L^{-2}}_{\mid \mathcal Y_u})=0$ for any $u \in U$ and $p_* (\mathbb L^{-2})$ is a free sheaf on $U$ and the formation of $p_{\ast}$ commute with base extension 
	(see \cite[Chapitre III, \S 7]{EGA} or \cite[Chapter 0, \S 5]{MFK}).

	\smallskip
	\noindent
	Let $\mathbf P(p_* (\mathbb L^{-2})) \to U$ be the associated projective bundle.
	On $\mathcal Y \times_U \mathbf P(p_* (\mathbb L^{-2}))$, we consider the divisor $\mathcal B =\{(y, [r]) \, | \, r \in H^0(\mathcal Y_u, {\mathbb L^{-2}}_{\mid \mathcal Y_u}), y \in \mathcal Y_u, r(y)=0 \}$. Since on any fiber the line bundle associated to  $\mathcal B_{\mid \mathcal Y_u \times \{[r]\}}$ is ${\mathbb L^{-2}}_{\mid \mathcal Y_u}$, shrinking $U$ we can assume that $\mathcal O_{\mathcal Y \times_U \mathbf P(p_*( \mathbb L^{-2}))}(\mathcal B)=q^{\ast}\mathbb L^{-2}$, where $q: \mathcal Y \times_U \mathbf P(p_* (\mathbb L^{-2})) \xrightarrow{} \mathcal Y$ is the projection. 
	The divisor $\mathcal B$ is the zero locus of some section $\mathbf r \in H^0(q^{\ast}\mathbb L^{-2})$.
	
	\smallskip
	\noindent
	Let $t \in H^0(q'^{\ast}q^{\ast}\mathbb L^{-2})$ denote the tautological section on the total space of $q^{*} \mathbb L^{-1}$, where $q': q^{\ast}\mathbb L^{-1} \to  \mathcal Y \times_U \mathbf P(p_*(\mathbb L^{-2}))$ is the projection. Then we can construct a relative double covering
	\begin{equation*}\label{universal.double.cover}
	\xymatrix@C-25pt{
		(t^2- (q'^{\ast}\mathbf r))_{0}\ar@{=}[r] & \mathcal X \ar[d] \ar@{^{(}->}[rrrr]& && &
		q^{*} \mathbb L^{-1} \ar[lllld]\\
		& \SY \times_U \PP(p_* (\mathbb L^{-2}))&&&&,
	}
	\end{equation*}
	whose fiber over any point $(u, [r]) \in \mathbf P(p_* (\mathbb L^{-2})) $, with $u \in U$ and $r \in H^0({\mathbb L^{-2}}_{\mid \mathcal Y_u})$, is the double covering $\mathcal X_{(u,[r])} \to \mathcal Y_u$ defined by the divisor $\mathcal B_{\mid \mathcal Y_u \times \{[r]\}}=(r)_0 \in | {\mathbb L^{-2}}_{\mid \mathcal Y_u}|$. In fact, we restrict the construction to the open set $V \subset \mathbf P(p_*( \mathbb L^{-2}))$, where the divisors $\mathcal B_{\mid \mathcal Y_u \times \{[r]\}}$ are reduced and smooth,
	in order to obtain integral, smooth, double coverings 
	$\mathcal X_{(u,[r])} \to \mathcal Y_u$. The open set $V$ 
	contains the point $(u_0, [r_0])$ that corresponds to the 
	pair $(Y,B)$, so we can assume
	$V$ maps surjectively onto $U$.
	
	\smallskip
	\noindent
	Let $\Phi$ denote the composite map
	\begin{equation*}
	\mathcal X \to \mathcal Y \times_{U} V \hookrightarrow \mathbf P_{V}^N.
	\end{equation*}
	Then $\Phi$ is an algebraic deformation of $X \xrightarrow{\varphi} \mathbf P^N$.

	\smallskip
	\noindent
	We will compute the tangent space to $V$ at $(u_0,[r_0])$ by considering the fiber at $(u_0,[r_0])$ of the sequence
	\begin{equation*}
	0 \to \ST_{V/U} \to \ST_V \to \tau^{*} \ST_U \to 0,
	\end{equation*}
	associated to the projection $V \xrightarrow{\tau} U$.
	This way, since $H^1(\SO_Y)=0$, we obtain the sequence
	\begin{equation}\label{tangent.V}
	0 \to H^0(\SO_B(B)) \to \ST_V(u_0,[r_0]) 
	\xrightarrow{\mathrm{d}\tau} H^0(\SN_{Y/\PP^m}) \to 0.
	\end{equation}
	
	\noindent
	Comparing \eqref{tangent.V} to \eqref{tangent.D} we see that the Kodaira--Spencer map for the family $(\mathcal Y \times_U V \to V, \mathcal B, (u_0, [r_0]) \in V)$ is bijective, so this family is an algebraic formally semiuniversal deformation of $(Y \subset \mathbf P^N, B \in | \mathcal L^{-2} |)$. 
	
	\smallskip
	\noindent
	Since the differential $\mathrm{d}F$ is an isomorphism then the Kodaira--Spencer map for the family $(\mathcal X \xrightarrow{\Phi} \mathbb P_V^N \to V, (u_0, [r_0])\in V)$ is bijective, as well. Therefore $(\mathcal X \xrightarrow{\Phi} \mathbb P_V^N \to V, (u_0, [r_0])\in V)$ is an algebraic formally semiuniversal deformation for $\varphi$, which shows that any deformation 
	of $\varphi$ is a (finite)  morphism of degree $2$.
\end{proof}

\begin{proposition}
\label{prop.algebraic.deformation}
Let $\mathfrak X$ be a smooth, algebraic projective variety, 
let 
\begin{equation*}
 \phi:  \mathfrak X \longrightarrow \PP^N
\end{equation*}
be a morphism and let $L$ be a polarization on $\mathfrak X$.
If 
 $H^2(\mathcal O_\mathfrak X)=0$,
then $\varphi$ and $(\mathfrak X, L)$ have an algebraic 
formally semiuniversal deformation. 
\end{proposition}

\begin{proof}
Since $\mathfrak X$ is a projective 
variety, it has a formal 
versal deformation (see e.g. \cite{DolIsk}
or \cite{Zariski}).
Since $H^2(\mathcal O_\mathfrak X)=0$, 
then, by Grothendieck's existence theorem 
(see \cite[Theorem 2.5.13]{Sernesi}), 
this formal versal 
deformation of $\mathfrak X$ is effective. 
It follows 
from general deformation theory 
the existence of an algebraic formally versal 
(even semiuniversal) deformation
of 
 $\varphi$ 
 (see  \cite[Theorem 3.4.8]{Sernesi}).
In this case, 
the formal semiuniversal deformation of 
$\phi$ is also effective, so it 
is algebraizable by 
Artin's algebraization theorem (see 
\cite{Artin.alg}).

\smallskip
\noindent 
By \cite[Theorem 3.3.11.(i)]{Sernesi}, 
the functor Def$_{(X,L)}$ 
has a semiuniversal formal element. Then arguing as above, 
this element is algebraizable, so
there exists an algebraic formally semiuniversal 
deformation for $(X,L)$. 
\end{proof}

\noindent We give now a generalization of 
\cite[Lemma 2.4]{deformation}. 
We recall before the 
definition of Fano variety and variety of general type:

\begin{definition}
 {\rm 
 Let $\mathfrak X$  be a smooth variety of dimension 
 $m$, $m \geq 2$. We say $\mathfrak X$ is a 
 Fano variety if $-K_{\mathfrak X}$ is ample. We say 
 that $\mathfrak X$ is of general type if $K_X$ is big.}
\end{definition}

\begin{lemma}\label{lemma.deformation.multiples.canonical}
 Let $\mathfrak X$  be a smooth variety of dimension $m$, $m \geq 2$, 
 which is 
either a 
 Fano variety
 or a 
 regular variety of general type. Let $l \in \NN$.
 \begin{enumerate}
  \item  If $\mathfrak X$ is Fano, let $L=\omega_\mathfrak X^{-l}$; 
  \item If $\mathfrak X$ is a variety of general type, let 
  $L=\omega_\mathfrak X^{\otimes l}$. 
 \end{enumerate}
 Assume $L$ is base--point--free. 
 Let $\phi$ be the morphism induced by $|L|$.
Then, for any deformation $\Phi: \mathcal X 
 \longrightarrow \PP^N_Z$ of $\varphi$, 
the morphism $\Phi_z$ is induced by $|\omega_{\mathcal X_z}^{-l}|$ 
if $\mathfrak X$ is Fano and by 
$|\omega_{\mathcal X_z}^{\otimes l}|$ 
if $\mathfrak X$ is of general type. 
\end{lemma}

\begin{proof}
The fact that $\Phi^*\mathcal O_{\mathbf P_z^N}(1)=
\omega_{\mathcal X_z}^{-l}$ if $\mathfrak X$ is Fano and 
$\Phi^*\mathcal O_{\mathbf P_z^N}(1)=
\omega_{\mathcal X_z}^{\otimes l}$ 
if $\mathfrak X$ 
is of general type follows from the same argument used in the proof 
of \cite[Lemma 2.4]{deformation}.

\smallskip
\noindent
If $\mathfrak X$ is Fano, then $H^i(L)=
H^i(\omega_\mathfrak X \otimes \omega_\mathfrak X^{-1} 
\otimes L)=0$ 
by Kodaira vanishing, 
for $\omega_\mathfrak X^{-1}$ and $L$ are
ample.
Then semicontinuity implies $h^0(L_z)$ is constant. 
If $\mathfrak X$ is of general type we use that the plurigenera are 
deformation invariants.
Shrinking $Z$ if necessary we may assume 
$L_z$ is base--point--free for all $z \in Z$. 
Then $\Phi_z$ is induced by the complete linear series of $|L_z|$.
\end{proof}

\section{Deformations of hyperelliptic Fano-K3 
varieties and generalized hyperelliptic polarized Fano varieties}\label{Section.Fano-K3}

\begin{notation}\label{notation.Y}
{\rm 
Unless otherwise stated, in the remaining of this article the variety
$Y$ has dimension $m$,  $m \geq 2$,
and is isomorphic to a variety of  one of these three types: 
\begin{enumerate}
 \item[(i)] $\mathbf P^m$.
 \item[(ii)] A smooth hyperquadric of dimension $m \geq 3$; 
 in this case, let $\mathfrak h$ be the restriction 
 of the hyperplane section of $\PP^{m+1}$ to $Y$. 
 \item[(iii)] A (smooth) projective bundle on $\PP^1$.
\end{enumerate}}
\end{notation}

\noindent
In order to study the deformations of hyperelliptic Fano-K3 varieties
we need first to carry 
out certain cohomology computations. We do them in the broader setting
of 
generalized hyperelliptic polarized Fano varieties.

\begin{proposition}\label{theorem.Fano}
Let $X$, $Y$ and $\mathcal E$ be as in notations \ref{setup} 
and \ref{notation.Y}, let $\mathcal T_Y$ be the tangent bundle of $Y$
and assume 
$X$ is Fano.
If $Y$ is a hyperquadric, let 
$\mathcal E \neq \mathcal O_Y(-\mathfrak h)$ and 
$\mathcal E \neq \mathcal O_Y(-2\mathfrak h)$. 
 Then $h^1(\mathcal T_Y \otimes \mathcal E)=0$. 
\end{proposition}

\begin{proof}
\emph{Case 1: $Y=\PP^m$.}
If $m >2$, 
then the vanishing of $H^1(\mathcal T_Y \otimes \mathcal E)$ follows from suitably twisting and
taking cohomology on the Euler sequence of tangent bundle of 
$\PP^m$ and by the vanishing of the intermediate
cohomology of line bundles on $\mathbf P^m$. 
If $m=2$, we also need to 
check that $H^2(\mathcal E)=0$. Recall 
$\omega_X=\pi^*(\omega_Y \otimes \mathcal E^{-1})$. 
Since $-K_X$ is ample, so is 
 $\omega_Y \otimes \mathcal E^{-1}$. 
 Then the degree of $\mathcal E$ is greater than or equal to $-2$, so
 $H^2(\mathcal E)$ indeed vanishes.

 \smallskip
 \noindent \emph{Case 2: $Y$ is a hyperquadric.}
 By Lefschetz hyperplane theorem the 
 Picard group of $Y$ is generated by $\mathcal O_Y(\mathfrak h)$.
Then $\mathcal E=\mathcal O_Y(\delta\mathfrak h)$, 
with $\delta$ a negative integer, for $X$ is connected. 
Then, by hypothesis $\delta \leq -3$. 
 Consider the exact sequence
 \begin{equation}\label{eq.normal.sequence.hyperquadric}
  0 \longrightarrow 
  \mathcal T_Y \otimes \mathcal E 
  \longrightarrow 
  \mathcal T_{\PP^{m+1}}|_Y \otimes \mathcal E \longrightarrow
  \mathcal O_Y(2\mathfrak h) \otimes \mathcal E 
  \longrightarrow 0, 
 \end{equation}
 obtained by tensoring the normal sequence of $Y$ in 
 $\PP^{m+1}$ by $\mathcal E$. 
 We want to see 
\begin{equation}\label{eq.vanishing.H1.restricted.tangent.hyperquadric}
 H^1(\mathcal T_{\PP^{m+1}}|_Y \otimes \mathcal E)=0.
\end{equation}
Consider the exact sequence
\begin{equation}\label{eq.restriction.Euler.hyperquadric}
 0 \longrightarrow \mathcal E \longrightarrow
 \mathcal O_Y(\mathfrak{h})^{m+2} \otimes \mathcal E 
 \longrightarrow \mathcal T_{\PP^{m+1}}|_Y \otimes \mathcal E
 \longrightarrow 0, 
\end{equation}
which is obtained by restricting to $Y$ the 
Euler sequence of the tangent bundle of $\mathbf P^{m+1}$
and tensoring with $\mathcal E$. Taking cohomology
on \eqref{eq.restriction.Euler.hyperquadric} we get
\begin{equation*}
 H^1(\mathcal O_Y((\delta +1)\mathfrak{h}))^{\oplus m+2}
 \longrightarrow 
 H^1(\mathcal T_{\PP^{m+1}}|_Y \otimes \mathcal E)
 \longrightarrow H^2(\mathcal O_Y(\delta \mathfrak{h})). 
\end{equation*}
Then we need $H^1(\mathcal O_Y((\delta +1)\mathfrak{h}))=0$
and 
\begin{equation}\label{eq.H^2.E.hyperquadric}
 H^2(\mathcal O_Y(\delta \mathfrak{h}))=0.
\end{equation}
Both vanishings follow
from suitably twisting and taking cohomology on the exact sequence
\begin{equation}
\label{eq.presentation.structure.sheaf.hyperquadric}
 0 \longrightarrow \mathcal O_{\PP^{m+1}}(-2)  
 \longrightarrow \mathcal O_{\PP^{m+1}} 
 \longrightarrow \mathcal O_Y
 \longrightarrow 0, 
\end{equation}
because of the vanishing of the intermediate cohomology
of line bundles on $\PP^{m+1}$ (recall $m+1 \geq 4$). 

\smallskip
\noindent
Now we study 
$H^0(\mathcal E \otimes \mathcal O_Y(2\mathfrak h))=
H^0(\mathcal O_Y((\delta+2)\mathfrak h))$. 
If $\delta \leq -3$, 
then 
$H^0(\mathcal O_Y((\delta+2)\mathfrak h))=0$. This together
with \eqref{eq.normal.sequence.hyperquadric}
and \eqref{eq.vanishing.H1.restricted.tangent.hyperquadric}
yields
$H^1(\mathcal T_Y \otimes \mathcal E)=0$ if $Y$ is a 
hyperquadric.

\smallskip
 \noindent \emph{Case 3: $Y$ is a projective bundle on $\PP^1$.}
 We use the following notation:
 \begin{equation}\label{notation.Y.Fano}
  \end{equation}
 \begin{enumerate}
  \item[(i)]  Let $Y=\PP(E_0)$ with $E_0$ normalized, and let 
\begin{equation*}
 E_0=\mathcal O_{\PP^1} \oplus \mathcal O_{\PP^1}(-e_1) \oplus \cdots  
 \mathcal O_{\PP^1}(-e_{m-1}), 
\end{equation*}
with $0 \leq e_1 \leq \cdots \leq e_{m-1}$. 
  \item[(ii)] Let $H_0$ be such that $\mathcal O_Y(H_0)=\mathcal O_{\PP(E_0)}(1)$. 
  \item[(iii)] 
 Let $p$ be the structure morphism from $Y$ to $\PP^1$ and let $F$ be a fiber 
 of $p$.
 \item[(iv)] Let $B \sim 2\alpha H_0 + 2\beta F$
 \end{enumerate}
Recall $K_Y=-mH_0-(e_1+\cdots +e_{m-1}+2) F$. Recall $\omega_X=
 \pi^*(\omega_Y \otimes \mathcal E^{-1})$. Since $\omega_X^{-1}$ is ample, 
 so is 
 $\omega_Y \otimes \mathcal E^{-1}$, therefore $\alpha < m$.

 \smallskip
 \noindent   Let $\mathcal T_{Y/\PP^1}$ be the relative tangent 
 bundle to $p$ and consider the exact sequence
 \begin{equation}\label{eq1}
  0 \longrightarrow \mathcal T_{Y/\PP^1} \longrightarrow \mathcal T_Y 
  \longrightarrow p^*\mathcal T_{\PP^1} \longrightarrow 0
  \end{equation}
and the relative Euler sequence
\begin{equation}\label{eq2}
 0 \longrightarrow \mathcal O_Y \longrightarrow p^*E_0^\vee \otimes \mathcal O_Y(H_0) \longrightarrow
 \mathcal T_{Y/\PP^1} \longrightarrow 0. 
\end{equation}

 \smallskip
 \noindent 
First we argue for $\alpha >1$.
We see that $H^1(p^*\mathcal T_{\PP^1}\otimes \mathcal E)=
H^1(\mathcal O_Y(-\alpha H_0 + (2 -\beta) F))=0$. Indeed, 
\begin{equation*}
 H^1(\mathcal O_{\PP^{m-1}}(-\alpha)) = \cdots = H^{m-2}(\mathcal O_{\PP^{m-1}}(-\alpha))=
  H^{m-1}(\mathcal O_{\PP^{m-1}}(-\alpha))=0,
\end{equation*}
because of the vanishing of the intermediate cohomology of $\mathbf P^{m-1}$ and, 
for the topmost cohomology, because $\alpha \leq m-1$. 
Then 
$R^ip_*(\mathcal O_Y(-\alpha H_0 + (2 -\beta) F))=0$ for all $i > 0$, so, by the Leray's spectral sequence, 
$H^1(\mathcal O_Y(-\alpha H_0 + (2 -\beta) F))=
H^1(p_*\mathcal O_Y(-\alpha H_0 + (2 -\beta) F))$. 
On the other hand, 
$H^0(\mathcal O_{\PP^{m-1}}(-\alpha))=0$ for $\alpha >0$, 
so $p_*\mathcal O_Y(-\alpha H_0 + (2 -\beta) F)
=0$, so $H^1(p_*\mathcal O_Y(-\alpha H_0 + (2 -\beta) F))=0$.

\smallskip
\noindent We see that $H^1(p^*E_0^{\vee} \otimes \mathcal O_Y(H_0) \otimes \mathcal E)=0$. 
Indeed, 
$p^*E_0^\vee \otimes \mathcal O_Y(H_0) \otimes \mathcal E$ is a direct sum of $m$ line bundles on $Y$, each of them 
of the form $\mathcal O_Y((1-\alpha)H_0+\delta_j F)$, 
for certain $\delta_j \in \ZZ$, $1 \leq j \leq 
m$. Since $\alpha > 1$, we can argue as above. 
We see that $H^2(\mathcal E)=H^2(\mathcal O_Y(-\alpha H_0 -\beta F))=0$ if $\alpha >0$ using the same arguments. 
Then exact sequences \eqref{eq1} and \eqref{eq2} and the vanishing of 
$H^1(p^*\mathcal T_{\PP^1}\otimes \mathcal E)$, $H^1(p^*E_0^\vee \otimes \mathcal O_Y(H_0) \otimes \mathcal E)$ 
and $H^2(\mathcal E)$ imply $H^1(\mathcal T_Y \otimes \mathcal E)=0$.

\smallskip
\noindent Now we argue for $\alpha=1$. Note that we have already proved
$H^1(p^*\mathcal T_{\PP^1}\otimes \mathcal E)=
H^2(\mathcal E)=0$ if $\alpha >0$. Thus we 
only need to compute $H^1(p^*E_0^\vee \otimes \mathcal O_Y(H_0) \otimes \mathcal E)$. 
Recall 
\begin{equation*}
 \omega_Y^{-1} \otimes  \mathcal E=
\mathcal O_Y((m-1)H_0+(e_1+\cdots + e_{m-1} + 2 - \beta)F)
\end{equation*}
is ample. This is equivalent to $e_1 + \cdots + e_{m-1} + 2 - \beta > (m-1)e_{m-1}$. 
Then $(m-1)e_{m-1} + 2 - \beta > (m-1)e_{m-1}$, so $\beta <2$. On the other hand
\begin{equation*}
 \begin{matrix}
 H^1(p^*E_0^\vee \otimes \mathcal O_Y(H_0) \otimes \mathcal E)=H^1(\mathcal O_Y(-\beta F)) \oplus 
 H^1(\mathcal O_Y((e_1-\beta) F)) \oplus \cdots \oplus 
 H^1(\mathcal O_Y((e_{m-1}-\beta) F))= \\
 H^1(\mathcal O_{\PP^1}(-\beta)) \oplus 
 H^1(\mathcal O_{\PP^1}(e_1-\beta)) \oplus \cdots \oplus 
 H^1(\mathcal O_{\PP^1}(e_{m-1}-\beta))=0 
 \end{matrix}
\end{equation*}
if $\beta <2$. Then exact sequences \eqref{eq1} and \eqref{eq2} and the vanishing of 
$H^1(p^*\mathcal T_{\PP^1}\otimes \mathcal E)$, $H^1(p^*E_0^\vee \otimes \mathcal O_Y(H_0) \otimes \mathcal E)$ 
and $H^2(\mathcal E)$ imply $H^1(\mathcal T_Y \otimes \mathcal E)=0$.

\smallskip
\noindent Finally we argue for $\alpha=0$. 
Now 
\begin{equation*}
 \omega_Y^{-1} \otimes  \mathcal E=
\mathcal O_Y(mH_0+(e_1+\cdots + e_{m-1} + 2 - \beta)F)
\end{equation*} 
is ample. This is equivalent to $e_1 + \cdots + e_{m-1} + 2 - \beta > me_{m-1}$. 
Then $(m-1)e_{m-1} + 2 - \beta > me_{m-1}$, so $\beta <2-e_{m-1}$. In addition, $\beta \neq 0$, 
otherwise $X$ would be disconnected. Then $1 \leq \beta < 2-e_{m-1}$, so $e_{m-1}=0$
and $\beta=1$. Then $H^1(p^*\mathcal T_{\PP^1}\otimes \mathcal E)=H^1(F)=0$, 
$H^2(\mathcal E)=H^2(-F)=0$ and 
 $H^1(p^*E_0^\vee \otimes \mathcal O_Y(H_0) \otimes \mathcal E)=H^1(\mathcal O_Y(H_0-F))^{m}=
 H^1(\mathcal O_{\PP^1}(-1))^{m}=0$. 
 Then exact sequences \eqref{eq1} and \eqref{eq2} and the vanishing of 
$H^1(p^*\mathcal T_{\PP^1}\otimes \mathcal E)$, $H^1(p^*E_0^\vee \otimes \mathcal O_Y(H_0) \otimes \mathcal E)$ 
and $H^2(\mathcal E)$ imply $H^1(\mathcal T_Y \otimes \mathcal E)=0$.
\end{proof}

\noindent We now study the existence or non existence of 
 nonsplit Fano ribbons on varieties $Y$ as in  Notation~\ref{notation.Y}
 (for the definitions of ribbon and
nonsplit ribbon, see \cite[\S 1]{BE}; by a Fano ribbon   
 we mean that $h^0(\mathcal O_{\widetilde Y})=1$ and 
that the dual of the  
dualizing sheaf $\omega_{\widetilde Y}$ of $\widetilde Y$ is ample): 

\begin{corollary}\label{cor.Fano.ribbons}
Let $Y$ be as in Notation~\ref{notation.Y}.
There are no nonsplit Fano ribbons  $\widetilde Y$
on $Y$ except if $Y$ is a hyperquadric in $\PP^{m+1}$ and 
$\widetilde Y$ is the unique ribbon in $\PP^{m+1}$ supported on $Y$.
\end{corollary}

\begin{proof}
Let $\widetilde Y$ be a Fano ribbon supported on $Y$ and let 
$\widetilde{\mathcal E}$ be the conormal bundle of $Y$ in $\widetilde Y$. 
Since $\omega_{\widetilde Y}^{-1}$ is ample, by 
\cite[Lemma 1.4]{Enriques}, so is $\omega^{-1}_Y \otimes \widetilde{\mathcal E}$, so 
$\widetilde{\mathcal E}$ can be thought as the trace--zero module of a cover 
$\pi: X \longrightarrow Y$ as in Proposition~\ref{theorem.Fano}, except for the 
fact that $|\widetilde{\mathcal E}|$ might not contain a smooth divisor. 
Since we do not 
use this property of $\mathcal E$ in the proof of Proposition~\ref{theorem.Fano}
and since $h^0(\mathcal O_{\widetilde Y})=1$ translates into $X$ being connected, 
it follows from Proposition~\ref{theorem.Fano} that 
$\mathrm{Ext}^1(\Omega_Y, \widetilde{\mathcal E})=0$ except maybe if $Y$ is a hyperquadric and 
$\widetilde{\mathcal E}=\mathcal O_Y(-\mathfrak h)$. In this case 
$H^0(\mathcal T_{\PP^{m+1}} \otimes \widetilde{\mathcal E})=
H^0(\mathcal O_Y)^{\oplus m+2}$ and, 
taking cohomology on \eqref{eq.normal.sequence.hyperquadric}, we see that 
$H^1(\mathcal T_Y \otimes \widetilde{\mathcal E})$ is the cokernel of the map 
\begin{equation*}
 H^0(\mathcal O_Y(\mathfrak h)) \otimes H^0(\mathcal O_Y) \longrightarrow 
 H^0(\mathcal O_Y(\mathfrak h))
\end{equation*}
of multiplication of global sections, which is, trivially, an isomorphism. Thus 
$\mathrm{Ext}^1(\Omega_Y, \widetilde{\mathcal E})=0$ also if $Y$ is a hyperquadric and 
$\widetilde{\mathcal E}=\mathcal O_Y(-\mathfrak h)$. 
Then the result follows from 
\cite[Corollary 1.4]{BE}.
\end{proof}

\begin{corollary}\label{cor.Fano.Hom}
Let $X$, $Y$, $\varphi$, $\mathcal E$ and $\mathcal I$ be as in 
notations \ref{setup} and \ref{notation.Y}. 
Assume 
$X$ is Fano
and $\varphi$ is  induced by a complete linear series
(i.e., $H^0(\mathcal E(1))=0$). 
Then 
 $\mathrm{Hom}(\mathcal I/\mathcal I^2, \mathcal E)=0$, except if  $Y$ is a 
 hyperquadric and
 $\mathcal E = \mathcal O_Y(-2\mathfrak h)$, in which case 
 $\mathrm{Hom}(\mathcal I/\mathcal I^2, \mathcal E)$ has dimension $1$.
\end{corollary}

\begin{proof}
Taking cohomology on the conormal sequence of $i(Y)$ in $\PP^N$ we get 
\begin{equation}\label{eq.cor.Fano}
 \mathrm{Hom}(\Omega_{\PP^N}|_{i(Y)}, \mathcal E) 
 \overset{\gamma}\longrightarrow 
 \mathrm{Hom}(\mathcal I/\mathcal I^2, \mathcal E)
 \longrightarrow \mathrm{Ext}^1(\Omega_Y,\mathcal E)
 \longrightarrow 
 \mathrm{Ext}^1(\Omega_{\PP^N}|_{i(Y)}, \mathcal E). 
\end{equation} 
We are going to see $H^0(\mathcal T_{\PP^N}|_{i(Y)} \otimes \mathcal E)=0$. 
 For this we need $H^0(\mathcal E(1))=0$
 and $H^1(\mathcal E)=0$. The former holds because $\varphi$ 
 factors through $\pi$ and 
 $\varphi$ is induced by a complete linear series. 
 The latter 
 holds by Kodaira vanishing theorem, 
 because $\omega_X^{-1}=\pi^*(\omega_Y^{-1} \otimes \mathcal E)$ is ample, 
 hence so is $\omega_Y^{-1} \otimes \mathcal E$.

 \smallskip
 \noindent If $Y$ is not a hyperquadric, then
 the result follows from  \eqref{eq.cor.Fano}
 and Proposition~\ref{theorem.Fano}. 
 If $Y$ is a hyperquadric, 
 since $H^0(\mathcal E(1))=0$, then $\mathcal E=
 \mathcal O_Y(\delta \mathfrak h)$ with $\delta \leq -2$. 
If $\delta \leq -3$, then
 the result follows from  \eqref{eq.cor.Fano}
 and Proposition~\ref{theorem.Fano}.

 \smallskip
 \noindent If $Y$ is  a hyperquadric and 
 $\mathcal E = \mathcal O_Y(-2\mathfrak h)$, then 
 $H^0(\mathcal E(1))=0$ implies 
 $L=\pi^*\mathcal O_Y(\mathfrak h)$ and $i$ in the 
 embedding of $Y$ as a hyperquadric in $\PP^{m+1}$. 
 Then $\mathcal I/\mathcal I^2 \simeq 
 \mathcal O_Y(-2\mathfrak h)$, so 
 $\mathrm{Hom}(\mathcal I/\mathcal I^2,  \mathcal E)$ is 
 isomorphic to $H^0(\mathcal O_Y)$ in this case. 
 \end{proof}

\begin{theorem}\label{thm.Fano.deform}
{Let $X$, $Y$ and  $\varphi$ be as in 
notations \ref{setup} and \ref{notation.Y}. Let 
$X$ be a  Fano variety
and let the morphism $\varphi$ from $X$ to $\PP^N$ be  
induced by a complete linear series 
(i.e., $H^0(\mathcal E(1))=0$).
\begin{enumerate}
 \item If $Y$ is a hyperquadric, 
 assume  
 $\mathcal E \neq \mathcal O_Y(-2\mathfrak h)$.
 \item If $Y$ is a projective bundle over $\PP^1$, assume 
 $B$ is base--point--free.  
\end{enumerate}
Then $\varphi$ is unobstructed and any deformation of $\varphi$
is a finite morphism of degree
 $2$ onto its image, which is a deformation of $i(Y)$ in $\PP^N$.} 
\end{theorem}

\begin{proof}
 If $Y$ is as Notation~\ref{notation.Y}, then 
 it is well-known that $H^1(\mathcal O_Y)$, 
 $H^2(\mathcal O_Y)$ and 
 $H^1(\mathcal N_{i(Y),\PP^N})$ vanish. 
 Thus 
 hypotheses (1), (3) and (4) of 
 Theorem~\ref{Psi2=0.3} are satisfied. 
 If $Y$ is $\PP^m$ or a hyperquadric, then 
 $H^1(\mathcal E^{-2})=0$ because the vanishing of 
 cohomology in projective space.
 If $Y$ is a projective bundle over $\PP^1$, then 
 condition (2) of the statement is equivalent to 
 $\beta \geq \alpha e_{m-1}$, with 
 $\alpha, \beta$ and $e_{m-1}$ as in \eqref{notation.Y.Fano}. 
 Then condition (2) implies 
 $H^1(\mathcal E^{-2})=0$. Thus hypothesis (2) 
 of Theorem~\ref{Psi2=0.3} is also satisfied.
 Finally, hypothesis (5) of Theorem~\ref{Psi2=0.3}
 follows from Corollary~\ref{cor.Fano.Hom}.
\end{proof}

\noindent The cases dealt with in the next proposition 
have already been covered by Theorem~\ref{thm.Fano.deform} except maybe 
if assumption (2) of Theorem~\ref{thm.Fano.deform} does not hold; that
is why we include it here. 

\begin{proposition}\label{prop.Fano.deform.lK}
 Let $X$, $Y$ and  $\varphi$ be as in 
notations \ref{setup} and \ref{notation.Y}. Let 
$X$ be a  Fano variety
and let the morphism $\varphi$ from $X$ to $\PP^N$ be  
induced by a complete linear series $|L|$, where 
$L=\omega_X^{-l}$ for some $l \in \NN$.
Let $Y$ be a projective bundle over $\PP^1$.
Then any
deformation of $\varphi$
is a finite morphism of degree
 $2$ onto its image, which is a deformation of $i(Y)$ in $\PP^N$.
\end{proposition}

\begin{proof}
Let $Z$ be a smooth algebraic variety 
with a distinguished point $0$ and 
let $(\mathcal X, \Phi)$ be a flat family
over $Z$, with $(\mathcal X_0, \Phi_0)=(X, \varphi)$. 
Proposition~\ref{theorem.Fano} and \cite[Corollary 1.11]{Wehler} 
(see also \cite[Theorem 8.1]{Horikawa})   imply that 
$\mathcal X$ is equipped, over an analytic neighborhood 
$U$ of $0$ in $Z$, 
with a  deformation $(\mathcal X,\Pi)$ 
 of $(X,\pi)$, where $\Pi$
 is finite, surjective and of degree $2$
(we make an abuse of notation and keep calling 
 the restriction of the family $\mathcal X$ to $U$ as $\mathcal X$). 
 We call $\mathcal Y$ to the image of $\Pi$. 
 Since $H^1(\mathcal O_Y(1))=
 H^2(\mathcal O_Y)=0$, we have $H^1(\mathcal T_{\PP^N}|_Y)=0$. 
 This and 
 \cite[Theorem 8.1]{Horikawa} imply that 
the deformation $\mathcal Y$ of $Y$ can be realized as 
a deformation $(\mathcal Y, \mathfrak i)$ of $(Y,i)$. 
Since $i$ is an embedding, 
we can assume, after shrinking $U$, that $\mathfrak i$ 
is a relative 
embedding of $\mathfrak Y$ in $\mathbf P^N_U$.
Then $\mathfrak i \circ \Pi=\Phi'$ is a deformation of
$\varphi$.
 By Lemma~\ref{lemma.deformation.multiples.canonical}, 
after shrinking $U$  if necessary, for all $z \in U$ we have 
$\Phi_z$ and 
$\Phi'_z$ are induced by $|\omega_{\mathcal X_z}^{-l}|$. 
 Thus, for all $z \in U$, $\Phi_z$ and $\Phi'_z$ are 
equal up to automorphisms of $\PP^N$. Thus, shrinking $U$ again 
if necessary, 
$\Phi_z$ is of degree $2$ for all $z \in U$. 
\end{proof}

\begin{remark}\label{remark.Fano.Wehler}
 {\rm If $X$, $Y$ and  $\varphi$  are as in Theorem~\ref{thm.Fano.deform}
or Proposition~\ref{prop.Fano.deform.lK} and 
$Y$ is a projective bundle 
over $\PP^1$, but neither condition (2) of 
Theorem~\ref{thm.Fano.deform} nor 
the condition on $L$ in Proposition~\ref{prop.Fano.deform.lK} hold, still 
something can be said about the deformations of $\varphi$. 
Indeed, it 
follows from Proposition~\ref{theorem.Fano}, 
\cite[Corollary 1.11]{Wehler} 
and \cite[Theorem 8.1]{Horikawa} that, 
for any 
deformation $\mathcal X$ of $X$, there exists a 
deformation of $(\mathcal X,\Phi)$ of $(X,\varphi)$
which is finite and of degree $2$ onto its image.}
\end{remark}

%\smallskip
\noindent 
We give the definition of 
Fano-K3 polarized variety, which is
equivalent to \cite[Definition 1.5]{Fujita}
(note that the condition of the ring 
$R(L)= \oplus_{n=0}^\infty H^0(L^{\otimes n})$
being Cohen-Macaulay, which is required 
by \cite[Definition 1.5]{Fujita}, is deduced
from Definition~\ref{defi.Fano-K3}, Kodaira vanishing theorem and  
Serre duality). 
\begin{definition}\label{defi.Fano-K3}
{\rm We say that a polarized variety 
$(X,L)$ of dimension $m$, $m \geq 3$, is Fano-K3 if it is a Fano polarized 
variety of index $m-2$, i.e., if $\omega_X^{-1}=L^{\otimes m-2}$.}
\end{definition}
\begin{proposition}\label{prop.FanoK3.classification}
Let $X$, $Y$, $\mathcal E$, $i$ and $\varphi$ be 
as in Notation~\ref{setup}, \ref{notation.Y} and 
\eqref{notation.Y.Fano}, and assume $(X,L)$ is 
a hyperelliptic Fano-K3 variety of dimension 
$m \geq 3$ and that $\varphi$ is induced 
by $|L|$. Then $i(Y)$ is a variety of minimal degree and 
$Y$, $\mathcal E^{-2}$ and $B$ are as follows: 
\begin{enumerate}
 \item If $Y=\PP^m$, then $\mathcal E^{-2}=
 \mathcal O_{\PP^m}(6)$ (in this case, $g=2$).  
 \item If $Y$ is a hyperquadric, 
 then $B \sim
 4\mathfrak h$ (in this case, $g=3$). 
 \item If $Y$ is a projective bundle over $\PP^1$, 
 then 
 \begin{enumerate}
  \item $m=4$, $e_1=e_2=e_3=0$, $L \sim 
  H_0+F$ and $B \sim 4H_0$ (in this case, $g=5$); 
  \item $m=3$, $e_1=e_2=0$ and $L \sim 
  H_0+2F$ and $B \sim 4H_0$ (in this case, $g=7$); 
  \item $m=3$, $e_1=e_2=0$ and $L \sim 
  H_0+F$ and 
  $B \sim 4H_0+2F$ (in this case, $g=4$); 
  \item $m=3$, $e_1=e_2=-1$, $L \sim 
  H_0+2F$ and $B \sim 4H_0$ (in this case, $g=5$).
 \end{enumerate}
\end{enumerate}
\end{proposition}

\begin{proof}
 If $(X,L)$ a Fano-K3 variety of dimension $m$,
 then by adjunction $L^m=2g-2$. Now the result 
 follows from 
 \cite[Proposition 5.18 (3), (6.7)]{Fujita}. 
\end{proof}

\noindent As a consequence of Theorem~\ref{thm.Fano.deform}
we obtain this:

\begin{theorem}\label{thm.Fano-K3.deform.hyperelliptic}
Let $X$ and $Y$ be as in 
notations \ref{setup} (2). 
 Let $(X,L)$ a hyperelliptic 
 Fano-K3 variety of dimension $m \geq 3$.
If $Y$ is $\PP^m$ or a rational normal scroll, 
then any 
deformation of $(X,L)$  
is 
 hyperelliptic. 
\end{theorem}

\begin{proof}
If $Y=\PP^m$, the result follows from 
Remark~\ref{remark.deformation.varieties} (2). 
Let now $Y$ be a projective bundle over $\PP^1$ and 
let $\varphi$ be the morphism induced by $|L|$. 
Let $Z$ be a smooth, algebraic variety with a distinguished point
$0 \in Z$.
Let $(\mathcal X, \mathcal L)$ be a flat family  
over $Z$ such that the fiber $(\mathcal X_0, \mathcal L_0)$ 
over $0$
is isomorphic to $(X,L)$. Recall that $H^1(L)=H^1(\pi_*L)$ and that, 
by the projection formula, the latter equals 
$H^1(\mathcal O_Y(1))
\oplus H^1(\mathcal E(1))$. We have 
$H^1(\mathcal O_Y(1))=0$;  
$H^1(\mathcal E(1))$ also vanishes because of the ampleness 
of $-K_X$, hence $H^1(L)=0$. Then, 
shrinking $Z$ if necessary, by semicontinuity  
$H^1(\mathcal L_z)=0$ for all $z \in Z$. 
Then, shrinking $Z$ if necessary, the push--out 
$(\mathcal X \to Z)_* (\mathcal L)$ is a free sheaf on $Z$ 
and the formation of $(\mathcal X \to Z)_{\ast}$ commute 
with base extension (see \cite[Chapitre III, \S 7]{EGA} 
or \cite[Chapter 0, \S 5]{MFK}). Thus $h^0(\mathcal L_z)$
is the same for all $z \in Z$, so
we obtain from $\mathcal L$  a $Z$--morphism
\begin{equation*}
 \Phi: \mathcal X \longrightarrow \PP^N_Z
\end{equation*}
such that $\Phi_0=\varphi$, i.e., $\Phi$ is a deformation of 
$\varphi$. 

\smallskip 

\noindent Now we apply Theorem~\ref{thm.Fano.deform} 
to $\Phi$. By Proposition~\ref{prop.FanoK3.classification} (3), 
the line bundle 
 $\mathcal E^{-2}$ satisfies condition (2) 
 of Theorem~\ref{thm.Fano.deform}. 
 Then 
Theorem~\ref{thm.Fano.deform}
 implies that, for all $z \in Z$, the morphism $\Phi_z$
 has degree $2$ onto 
 its image, which is a deformation of $i(Y)$ in $\PP^N$. 
Since 
 any deformation of $i(Y)$ is a variety of minimal degree, 
 the morphism $\Phi_z$ is induced by the complete linear
 series $H^0(\mathcal L_z)$ and $\mathcal L_z^m=L^m$, 
 we conclude that 
 $(\mathcal X_z, \mathcal L_z)$ is a hyperelliptic polarized 
 variety. 
 \end{proof}

 \begin{theorem}\label{thm.Fano-K3.deform.nonhyperelliptic}
 Let $X$ and $Y$ be as in 
notations \ref{setup} (2).   
 Let $(X,L)$ a hyperelliptic 
 Fano-K3 variety of dimension $m \geq 3$  
and 
 let $\varphi$ be the morphism induced by 
 the complete linear series $|L|$.
If $Y$ is a hyperquadric, then we have:
\begin{enumerate}    
 \item The morphism $\varphi$ and the polarized variety $(X,L)$
 are
 unobstructed; 
\item A general deformation of $\varphi$ is an embedding. 
Likewise, 
a general deformation of $(X,L)$ is nonhyperelliptic
but its complete linear series induces an embedding. 
The images of those embeddings are quartic hypersurfaces
 of $\mathbf P^{m+1}$.
\end{enumerate}
 \end{theorem}

\begin{proof} 
 If $Y$ is a hyperquadric, then 
$\mathcal E = \mathcal O_Y(-2\mathfrak{h})$ by 
Proposition~\ref{prop.FanoK3.classification} (2) 
and $\mathcal O_Y(1)=
\mathcal O_Y(\mathfrak{h})$. Since $H^2(\mathcal O_Y)=0$ 
and
$H^2(\mathcal E)=0$ by \eqref{eq.H^2.E.hyperquadric}, then 
$H^2(\mathcal O_X)=0$,
so, by Proposition~\ref{prop.algebraic.deformation}, 
there exist
algebraic formally semiuniversal deformations 
 of
$\varphi$ and $(X,L)$.

\smallskip
\noindent
If follows from \cite[(2.11)]{deformation} that 
$H^1(\mathcal N_\pi)$ is isomorphic to $H^1(\mathcal O_B(B))$.   
To compute $H^1(\mathcal O_B(B))$ we consider the  sequence
\begin{equation}\label{eq.compute.H1.Npi}
H^1(\mathcal O_Y(B)) 
\longrightarrow H^1(\mathcal O_B(B)) 
\longrightarrow H^2(\mathcal O_Y).
\end{equation}
Note that, by the Kodaira vanishing theorem, 
$H^1(\mathcal O_Y(B))=H^1(\mathcal O_Y(4\mathfrak{h}))=0$ and
$H^2(\mathcal O_Y)=0$, so  $H^1(\mathcal O_B(B))$ and 
$H^1(\mathcal N_\pi)$
vanish. 
From 
\eqref{eq.normal.bundles.cohomology} and 
from the vanishing of $H^1(\mathcal O_Y)$ and 
$H^1(\mathcal O_Y(2\mathfrak{h}))$, it follows
that 
 $H^1(\mathcal N_\varphi)=0$,
so $\varphi$ 
is unobstructed.
Now 
let $\Sigma_L$ be the sheaf of first order differential 
 operators in $L$ 
and let  
\begin{equation}\label{Atiyah.L}
0 \longrightarrow \SO_X \longrightarrow \Sigma_{L}
\longrightarrow \Cal{T}_{X}\longrightarrow 0
\end{equation}
be the Atiyah extension of $L$ 
(see  \cite[p. 96]{HarrisMorrison} or 
\cite[(3.30)]{Sernesi}). 
By \cite[Theorem 3.3.11 (ii)]{Sernesi},  
if $H^2(\Sigma_L)=0$, then $(X,L)$ is unobstructed. To prove
the vanishing of $H^2(\Sigma_L)$, 
by exact sequence  \eqref{Atiyah.L} 
and the vanishing of $H^2(\mathcal O_X)$, we only need to show 
$H^2(\mathcal T_X)=0$. For that 
we take cohomology on the exact sequence
\begin{equation}\label{eq.presentation.conormal.bundle}
 0 \longrightarrow \mathcal T_X \longrightarrow 
 \varphi^*\mathcal T_{\PP^N} \longrightarrow 
 \mathcal N_\varphi \longrightarrow 0. 
\end{equation}
By the vanishing of $H^1(\mathcal N_\varphi)$, 
the projection formula 
and the Leray's spectral sequence, it is enough 
to show the vanishings of 
$H^2(\mathcal T_{\PP^N}|_{i(Y)})$ and 
$H^2(\mathcal T_{\PP^N}|_{i(Y)} \otimes \mathcal E)$. 
For that we use the restriction to $i(Y)$ of the 
Euler sequence of the tangent bundle of $\PP^N$, so the
wanted 
vanishings will follow from the vanishings
of $H^2(\mathcal O_Y(\mathfrak h))$, 
$H^3(\mathcal O_Y)$, 
$H^2(\mathcal O_Y(-\mathfrak h))$ and 
$H^3(\mathcal O_Y(-2\mathfrak h))$. All those vanishings
follow from the vanishing of higher cohomology of line bundles
of $\PP^{m+1}$ (recall that $m \geq 3$). 
Therefore $H^2(\mathcal T_X)=0$ and so  $H^2(\Sigma_L)=0$
and $(X,L)$ is unobstructed. This completes the 
proof of (1).

\smallskip
\noindent
To prove (2) 
we are going to use \cite[Theorem 1.5]{smoothing}. 
By 
Corollary~\ref{cor.Fano.Hom}, 
$\mathrm{Hom}(\mathcal I/\mathcal I^2, \mathcal E)$ has 
dimension $1$. In fact, since 
$\mathcal I/\mathcal I^2=\mathcal E$, any nonzero element of 
$\mathrm{Hom}(\mathcal I/\mathcal I^2, \mathcal E)$ is an isomorphism. 
Recall the homomorphism $\Psi_2$, 
introduced in Proposition~\ref{morphism.miguel}, and cohomology 
sequence~\eqref{eq.normal.bundles.cohomology}. 
  Since $H^1(\mathcal N_\pi)$
vanishes,  
the homomorphism $\Psi_2$ 
 is surjective. Then,  given  a nonzero $\mu
 \in \mathrm{Hom}(\mathcal I/\mathcal I^2,
\mathcal E)$, there exists  
$\nu$ in 
$H^0(\mathcal N_\varphi)$ 
 such that $\Psi_2(\nu)=\mu$. 
Since there is an algebraic formally semiuniversal deformation 
 of $\varphi$ with base $Z$ and $\varphi$ is unobstructed, we see
that all the hypotheses 
of \cite[Theorem 1.5]{smoothing} are satisfied, 
so \cite[Theorem 1.5]{smoothing} and 
its proof imply that there exists a smooth, 
algebraic curve $T$ in $Z$, passing through 
 $0$ and tangent to the tangent vector  $v$ of $Z$ 
 corresponding to $\nu$ (recall that 
 $H^0(\mathcal N_\varphi)$ is isomorphic to the
 tangent space of $Z$ at $0$)
 and a deformation $\Phi_T$ of $\varphi$ 
 over $T$ 
 such that $\Phi_0=\varphi$ and 
 $\Phi_t$ is an embedding for all $t \in T, 
 t \neq 0$.
Taking the pullback by $\Phi_T$
 of $\mathcal O_{\PP^N_T}(1)$ gives us a 
 deformation $(\mathcal X_T, \mathcal L_T)$
 of $(X,L)$. Since 
 for all $t \in T \smallsetminus \{0\}$, $\Phi_t$ is 
 an embedding, then 
 $(\mathcal X_t, \mathcal L_t)$ is
 non hyperelliptic. Let $Z'$ be the base of an
 algebraic formally semiuniversal deformation 
 $(\mathcal X, \mathcal L)$
 of
 $(X,L)$.  Then, after shrinking $T$ if necessary,
 $(\mathcal X_T, \mathcal L_T)$ is obtained from 
 $(\mathcal X, \mathcal L)$
 by  
 etale base change, so 
 there is a point $z' \in Z'$ such that 
 $(\mathcal X_{z'}, \mathcal L_{z'})$ is nonhyperelliptic 
 but $\mathcal L_{z'}$ is very ample. 
Since very ampleness is an open condition, this 
finishes the proof of (2).
\end{proof}

\noindent 
Deformations of finite morphisms 
 of degree $2$ to embeddings 
 are linked to the existence of \emph{smoothable} 
   embedded ribbons (see e.g. \cite{Fong}, \cite{K3carpets}, 
   \cite{Gonzalez}, \cite{ropes}, 
   \cite{Enriques}, \cite{smoothing} or \cite{MR}). The next is another example
   of this phenomenon.

\begin{example}\label{example.smoothable.ribbon}
 {\rm Let $m \geq 3$. 
Given a hyperquadric $i(Y)$ in $\PP^{m+1}$, the dualizing
sheaf of 
the (unique) ribbon structure $\widetilde Y$ on $i(Y)$ 
embedded in $\PP^{m+1}$ 
is $\omega_{\widetilde Y}=\mathcal O_{\widetilde Y}(-m+2)$.
In addition, since $\widetilde Y$ is a divisor in 
$\PP^{m+1}$, it is arithmetically Cohen-Macaulay. Thus, 
we can say that $\widetilde Y$ is a Fano-K3 ribbon.

\smallskip
\noindent 
The ribbon $\widetilde Y$ can be \emph{smoothed}
inside $\PP^{m+1}$, i.e., there exists an embedded deformation of 
$\widetilde Y$ whose general fiber is a smooth (Fano-K3) variety
in $\PP^{m+1}$. This follows from the general theory developed 
in \cite[\S 1]{smoothing} but, in this case, can be 
achieved in an ad--hoc, more straight forward fashion,
by deforming $\widetilde Y$ in the linear
system of degree $4$ divisors on $\PP^{m+1}$.}
\end{example}

\begin{remark}
 {\rm In \cite[(7.2)]{Iskovskih} 
 Iskovskih classified  hyperelliptic $(X,L)$ 
 Fano--K3 threefolds. They are of five types and their 
 sectional genera  are 
 $g=2, 3, 4, 5$ and $7$ (compare with
 Proposition~\ref{prop.FanoK3.classification}). 
Iskovskih also proved 
 (see \cite[Theorem 3]{Debarre}) that, 
 if $X$ is a prime Fano threefold of genus $g$ and $g  \geq 4$, 
 then $-K_X$ is very ample.
 Proposition~\ref{prop.FanoK3.classification} and 
 Theorem~\ref{thm.Fano-K3.deform.hyperelliptic}
 say that
 hyperelliptic Fano--K3 threefolds with $g \geq 4$ deform only to  
 hyperelliptic Fano--K3 threefolds, 
so  
 for each 
genus $4$, $5$ and $7$ 
there are at least two kinds of Fano--K3 threefolds, hyperelliptic and 
anticanonically embedded, that do not deform to each other.}
\end{remark}

\section{Deformations of hyperelliptic 
polarized varieties with $L^m=2g$.}\label{Section.2g}

\begin{notation}\label{notation.Y.scroll}
 {\rm If $Y$ is  a projective bundle  
 embedded by $i$ as a rational normal scroll, then
 \begin{enumerate}
  \item[(i)] $E$ is the very ample vector bundle such that 
 $Y=\PP(E)$ and $i$ is induced by 
 $|\mathcal O_{\PP(E)}(1)|$; 
 \item[(ii)] 
 $E=\mathcal O_{\PP^1}(a_1) \oplus \cdots \oplus 
\mathcal O_{\PP^1}(a_n)$ with 
$a_1 \geq \cdots  \geq a_n > 0$ and denote
  $a=a_1+\cdots + a_n$ (thus 
 $i(Y)=S(a_1,\dots,a_n)$ and $N={a+n-1}$); 
 \item[(iii)]   $H$ will be the divisor 
corresponding 
to $\mathcal O_{\PP(E)}(1)$; 
\item[(iv)]  $p$ will be the structure morphism 
from $Y$ to $\PP^1$ and  
$F$ will be the fiber of $p$.
 \end{enumerate}}
\end{notation}

\begin{proposition}
 \label{prop.2g.classification}
Let $X$, $Y$, $\mathcal E$, $i$ and $\varphi$ be 
as in Notation~\ref{setup}, \ref{notation.Y}, 
\eqref{notation.Y.Fano} and 
\ref{notation.Y.scroll}. Assume $(X,L)$ is 
a hyperelliptic  variety of dimension 
$m \geq 2$ with $L^m=2g$ and  $\varphi$ is induced 
by $|L|$. Then $i(Y)$ is a variety of minimal degree and 
$Y$, $\mathcal E^{-2}$ and $B$ are as follows:
\begin{enumerate}
 \item If $Y=\PP^m$, then $\mathcal E^{-2}=
 \mathcal O_{\PP^m}(4)$ ($g=1$).
 \item If $Y$ is a projective bundle over $\PP^1$, 
 then 
 \begin{enumerate}
  \item  $B \sim 2H+2F$; or
  \item $m=3$, $a_1=a_2=a_3=1$ and 
  $B \sim 4H-4F$ (i.e., $B \sim 4H_0$); or 
  \item $m=2$, $a_1=a_2$ and $B \sim 4H - (4a_1-2)F$
  (i.e., $B \sim 4H_0 +2F$); or
  \item $m=2$, 
  $a_1=a_2+1$ and $B \sim 4H-4a_2F$ 
  (i.e., $B \sim 4H_0 + 4F$);  or
  \item $m=2$, 
  $a_1=a_2+2$ and $B \sim 4H-(4a_2+2)F$
  (i.e., $B \sim 4H_0+6F$). 
 \end{enumerate}
\end{enumerate}
\end{proposition}

\begin{proof}
 The result follows from \cite[Proposition 5.18 (2), (6.7)]{Fujita}. 
\end{proof}

\noindent We need  to carry out several 
cohomological computations on the varieties $X$ and $Y$ of
Proposition~\ref{prop.2g.classification}: 

\begin{proposition}\label{prop.2g.hom.not.0}
Let $X$, $Y$, $\mathcal E$, $i$, $\pi$  and $\varphi$ be 
as in Notation~\ref{setup}, \ref{notation.Y} and 
\ref{notation.Y.scroll}. Let $L=\varphi^*\mathcal O_Y(1)$.
Assume furthermore that 
$Y$ and $B$ are  as in 
Proposition~\ref{prop.2g.classification} (2a), (2c), (2d) or
(2e). 
Then $\varphi$ is induced by the complete linear series 
$|L|$,  the groups $\mathrm{Hom}
  (\mathcal I/\mathcal I^2, \mathcal E)$ and $\mathrm{Ext}^1
  (\Omega_Y, \mathcal E)$ are isomorphic,  and 
the following holds: 
 \begin{enumerate}
  \item  If $Y$ and $B$ are as in 
  Proposition~\ref{prop.2g.classification} (2a), then 
  the dimension of $\mathrm{Hom}
  (\mathcal I/\mathcal I^2, \mathcal E)$ is 
$g$;  otherwise, the dimension of $\mathrm{Hom}
  (\mathcal I/\mathcal I^2, \mathcal E)$ is
  $2$.
  \item $h^2(\mathcal O_X)=0$.
  \item $h^1(\mathcal N_\pi)=0$, except if $Y$ and $B$ are as in 
  Proposition~\ref{prop.2g.classification} (2e); in this
  case, $h^1(\mathcal N_\pi)=1$. 
\item $h^1(\mathcal N_\varphi)=0$; except if $Y$ and $B$ are as in 
  Proposition~\ref{prop.2g.classification} (2e); in this
  case, $h^1(\mathcal N_\varphi)=0$ or $1$.
 \end{enumerate}
\end{proposition}

\begin{proof}
We first see that $\varphi$ is induced by the complete linear
series $|L|$. This is equivalent to
\begin{equation}\label{eq.vanishing.for.complete.linear.series}
 H^0(\mathcal E(1))=0,
\end{equation}
so we check this in each case of 
Proposition~\ref{prop.2g.classification} (2a), (2c), (2d) or
(2e).
Indeed, if $Y$ and $B$ are as in 
Proposition~\ref{prop.2g.classification} (2a), (2c), (2d) or 
(2e), 
then $\mathcal E(1)$ equals   
\begin{equation}\label{eq.E(1)}
 \mathcal O_Y(-F), \,\mathcal O_Y(-H+(2a_1-1)F), 
\,\mathcal O_Y(-H+2a_2F) \ \mathrm{and} \ 
\mathcal O_Y(-H+(2a_2+1)F)
\end{equation}
respectively, and none of these line bundles have 
global sections.

\smallskip
\noindent
Now we prove $\mathrm{Hom}
  (\mathcal I/\mathcal I^2, \mathcal E)$ and $\mathrm{Ext}^1
  (\Omega_Y, \mathcal E)$ are isomorphic.
  For this we will prove 
that the connecting homomorphism $\gamma$ of
\eqref{eq.cor.Fano}
 is an isomorphism in this case and
then we will compute 
the dimension of $\mathrm{Ext}^1(\Omega_Y,\mathcal E)$, i.e.,  
$h^1(\mathcal T_Y \otimes \mathcal E)$.
To prove $\gamma$ is an isomorphism, we will show
\begin{equation}\label{eq.Euler.vanishings}
\mathrm{Hom}(\Omega_{\mathbf P^N}|_{i(Y)},\mathcal E)=
\mathrm{Ext}^1(\Omega_{\mathbf P^N}|_{i(Y)},\mathcal E)=0.
\end{equation}
 For that, we consider the restriction to $i(Y)$ of the Euler 
sequence of $\PP^N$ and get
\begin{equation*}
 \mathrm{Hom}(\mathcal O_{i(Y)}^{N+1}(-1),\mathcal E) 
 \longrightarrow \mathrm{Hom}(\Omega_{\mathbf P^N}|_{i(Y)},\mathcal E) 
 \longrightarrow \mathrm{Ext}^1(\mathcal O_Y,\mathcal E)
 \end{equation*}
 and 
 \begin{equation*}
 \mathrm{Ext}^1(\mathcal O_{i(Y)}^{N+1}(-1),\mathcal E) 
 \longrightarrow \mathrm{Ext}^1(\Omega_{\mathbf P^N}|_{i(Y)},\mathcal E) 
 \longrightarrow \mathrm{Ext}^2(\mathcal O_Y,\mathcal E).
 \end{equation*}
 It suffices to see that $\mathrm{Hom}(\mathcal O_{i(Y)}^{N+1}(-1),\mathcal E)$, 
 $\mathrm{Ext}^1(\mathcal O_{i(Y)}^{N+1}(-1),\mathcal E)$, 
 $\mathrm{Ext}^1(\mathcal O_Y,\mathcal E)$ and 
 $\mathrm{Ext}^2(\mathcal O_Y,\mathcal E)$ all vanish. 
On the one hand, 
$\mathrm{Hom}(\mathcal O_{i(Y)}^{N+1}(-1),\mathcal E)$ 
vanishes because of 
\eqref{eq.vanishing.for.complete.linear.series} 
and $\mathrm{Ext}^1(\mathcal O_{i(Y)}^{N+1}(-1),
\mathcal E)$ vanishes because none of the line bundles
of \eqref{eq.E(1)} has higher cohomology.
 On the other hand, if $Y$ and $B$ are as in 
 Proposition~\ref{prop.2g.classification} (2a), then 
 $\mathrm{Ext}^1(\mathcal O_Y,\mathcal E)$ and 
 $\mathrm{Ext}^2(\mathcal O_Y,\mathcal E)$ are isomorphic
 to $H^1(\mathcal O_Y(-H-F))$ and $H^2(\mathcal O_Y(-H-F))$ 
 and both also vanish. 
 If $Y$ is as in the other three cases, 
 $\mathrm{Ext}^1(\mathcal O_Y,\mathcal E)$ and 
 $\mathrm{Ext}^2(\mathcal O_Y,\mathcal E)$ are, 
 respectively, Serre dual of $H^1(\mathcal O_Y(-F))$ and 
 $H^0(\mathcal O_Y(-F))$ and both cohomology groups vanish. 
 This completes the proof of $\gamma$ being an isomorphism. 
 
 \smallskip
 
 \noindent Now we prove (1). We  
 compute the dimension $\mathrm{Ext}^1(\Omega_Y,\mathcal E)$, 
 which is isomorphic 
 to $H^1(\mathcal T_Y \otimes \mathcal E)$. 
In order to compute the dimension $H^1(\mathcal T_Y \otimes \mathcal E)$ 
 we consider the exact sequence \eqref{eq1}.
  If $Y$ and $B$ are as in Proposition~\ref{prop.2g.classification} (2a), 
  then $p^*\mathcal T_{\PP^1} \otimes \mathcal E=\mathcal O_Y(-H+F)$, 
both $H^0(p^*\mathcal T_{\PP^1} \otimes \mathcal E)$ and 
$H^1(p^*\mathcal T_{\PP^1} \otimes \mathcal E)$ vanish.
Then $H^1(\mathcal T_Y \otimes \mathcal E)$ is isomorphic to 
$H^1(\mathcal T_{Y/\PP^1} \otimes \mathcal E)$. 
In order to compute $H^1(\mathcal T_{Y/\PP^1} \otimes \mathcal E)$ 
 we consider the dual of the relative Euler sequence
\begin{equation}\label{eq4}
 0 \longrightarrow \mathcal O_Y \longrightarrow p^*E^\vee \otimes \mathcal O_Y(H) 
 \longrightarrow 
 \mathcal T_{Y/\PP^1} \longrightarrow 0. 
\end{equation}
Since $H^1(\mathcal O_Y(-H-F))=H^2(\mathcal O_Y(-H-F))=0$,
$H^1(\mathcal T_{Y/\PP^1} \otimes \mathcal E)$ is 
isomorphic to 
$H^1(p^*E^\vee\otimes \mathcal O_Y(H) \otimes \mathcal E)$, 
which is the same 
as
\begin{equation*}
 H^1(\mathcal O_Y(-(a_1+1)F)) \oplus \cdots  \oplus  H^1(\mathcal O_Y(-(a_n+1)F)), 
\end{equation*}
which has dimension $a_1 + \cdots + a_n=a=\frac{1}{2}L^m=g$. 
Then, if follows from all the above that
the dimension of Hom$(\mathcal I/\mathcal I^2,\mathcal E)$ 
is $g$ if $Y$ and $B$ are as in 
Proposition~\ref{prop.2g.classification} (2a). 

\smallskip 
\noindent Now we compute $h^1(\mathcal T_Y \otimes \mathcal E)$
when $Y$ and $B$ are as in Proposition~\ref{prop.2g.classification} (2c), 
(2d) or (2e). Since in this case $Y$ has dimension $2$,
with notation~\eqref{notation.Y.Fano}, exact sequence
\eqref{eq1} becomes
\begin{equation}\label{eq.tangent.bundle.Hirzebruch}
  0 \longrightarrow \mathcal O_Y(2H_0 + eF) \longrightarrow 
  \mathcal T_Y 
  \longrightarrow \mathcal O_Y(2F) \longrightarrow 0, 
  \end{equation}
  where $e=a_1-a_2$, i.e., $e=0$ in case (2c), 
$e=1$ in case (2d), $e=2$ in case (2e).
Tensoring with $\mathcal E$, \eqref{eq.tangent.bundle.Hirzebruch}
becomes 
\begin{equation}\label{eq.tangent.bundle.Hirzebruch.twisted}
  0 \longrightarrow \mathcal O_Y(-F) \longrightarrow 
  \mathcal T_Y \otimes \mathcal E
  \longrightarrow \mathcal O_Y(-2H_0+(1-e)F) \longrightarrow 0. 
  \end{equation}
 Note that $H^1(\mathcal O_Y(-F))=
  H^2(\mathcal O_Y(-F))=0$,
  so that $h^1(\mathcal T_Y \otimes \mathcal E)=
  h^1(\mathcal O_Y(-2H_0+(1-e)F))$. 
  By Serre duality, $h^1(\mathcal O_Y(-2H_0+(1-e)F))
  =h^1(\mathcal O_Y(-3F)=2$, 
so $h^1(\mathcal T_Y \otimes \mathcal E)$ and hence, 
the dimension of Hom$(\mathcal I/\mathcal I^2,\mathcal E)$,   
is $2$ if $Y$ and $B$ are as in 
Proposition~\ref{prop.2g.classification} (2c), 
(2d) or (2e).

\smallskip
\noindent Now we prove (2). By the projection formula and by 
the Leray spectral sequence, 
\begin{equation*}
 H^2(\mathcal O_X)= H^2(\mathcal O_Y) \oplus H^2(\mathcal E).
\end{equation*}
It is well known $H^2(\mathcal O_Y)=0$. 
If $Y$ and $B$ are as in Proposition~\ref{prop.2g.classification} (2a), 
then   
\begin{equation*}
 H^2(\mathcal E)=H^2(\mathcal O_Y(-H-F))=0.
\end{equation*}
If $Y$ and $B$ are as in Proposition~\ref{prop.2g.classification} (2c), 
(2d) or (2e), then 
\begin{equation*}
 H^2(\mathcal E)=
H^2(\mathcal O_Y(-2H_0-(e+1)F)),
\end{equation*}
 which, by Serre duality is 
isomorphic to $H^0(\mathcal O_Y(-F))^\vee$, which vanishes. Hence 
$H^2(\mathcal O_X)=0$ in all cases.

  \smallskip
\noindent Now we prove (3). 
  If follows from \cite[(2.11)]{deformation} that 
$H^1(\mathcal N_\pi)$ is isomorphic to $H^1(\mathcal O_B(B))$.   
To compute $H^1(\mathcal O_B(B))$ we consider the 
exact sequence \eqref{eq.compute.H1.Npi}.
Since $H^1(\mathcal O_Y)=H^2(\mathcal O_Y)=0$, 
the groups $H^1(\mathcal O_B(B))$ and $H^1(\mathcal O_Y(B))$ 
are isomorphic. 
Using the projection formula and the Leray spectral sequence, 
we see that $h^1(\mathcal O_Y(B))$ is the sum of the dimensions
of the first cohomology group of certain line bundles on 
$\PP^1$. If $Y$ and $B$ are as in 
Proposition~\ref{prop.2g.classification} (2a), then the smallest 
among the degrees of those line bundles is $2a_n+2$, which is 
greater than or equal to $4$; if $Y$ and $B$ are as in 
Proposition~\ref{prop.2g.classification} (2c), then the smallest 
among the degrees of those line bundles is $2$; and
if $Y$ and $B$ are as in 
Proposition~\ref{prop.2g.classification} (2d), then the smallest 
among the degrees of those line bundles is $0$. In all these
three cases the first cohomology groups of the line bundles
vanishes, so $H^1(\mathcal O_Y(B))$ vanishes and so do 
$H^1(\mathcal O_B(B))$ and $H^1(\mathcal N_\pi)$. If $Y$ and $B$ are as in 
Proposition~\ref{prop.2g.classification} (2e), then 
\begin{equation*}
 h^1(\mathcal O_Y(B))=h^1(\mathcal O_{\PP^1}(6)) \oplus
 h^1(\mathcal O_{\PP^1}(4)) \oplus
 h^1(\mathcal O_{\PP^1}(2)) \oplus
 h^1(\mathcal O_{\PP^1}) \oplus
 h^1(\mathcal O_{\PP^1}(-2))=1,  
\end{equation*}
so $h^1(\mathcal N_\pi)=h^1(\mathcal O_B(B))=
h^1(\mathcal O_Y(B))=1$.

\smallskip
\noindent Now we prove (4). 
 We use cohomology sequence~\eqref{eq.normal.bundles.cohomology}.
First we will prove the vanishing of
$H^1(\mathcal N_{Y,\mathbf P^N}) 
\oplus H^1(\mathcal N_{Y,\mathbf P^N} 
\otimes \mathcal E)$. 
The vanishing of $H^1(\mathcal N_{Y,\mathbf P^N})$ is well 
known.
To prove the vanishing of $H^1(\mathcal N_{Y,\mathbf P^N} 
\otimes \mathcal E)$, 
consider the normal sequence
\begin{equation*}
 0 \longrightarrow \mathcal T_Y \longrightarrow 
 \mathcal T_{\PP^N}|_{i(Y)} 
 \longrightarrow \mathcal N_{i(Y),\mathbf P^N} \longrightarrow 0. 
\end{equation*}
Then, for $H^1(\mathcal N_{Y,\mathbf P^N} 
\otimes \mathcal E)$ to vanish, it suffices to have 
the vanishings of
$H^1(\mathcal T_{\PP^N}|_{i(Y)} \otimes \mathcal E)$ 
 and 
$H^2(\mathcal T_Y \otimes \mathcal E)$.
The vanishing of $H^1(\mathcal T_{\PP^N}|_{i(Y)} 
\otimes \mathcal E)$ has been already 
proved (see \eqref{eq.Euler.vanishings}). 
For the vanishing of
$H^2(\mathcal T_Y \otimes \mathcal E)$ 
we use \eqref{eq1}, \eqref{eq4} and 
\eqref{eq.tangent.bundle.Hirzebruch.twisted}. 
We argue first for $Y$ and $B$ as in 
Proposition~\ref{prop.2g.classification} (2a). 
It suffices to prove the vanishings of 
$H^2(p^*\mathcal T_{\PP^1} \otimes \mathcal E)$, 
$H^2(p^*E^\vee \otimes \mathcal O_Y(H) \otimes \mathcal E)$ 
and 
$H^3(\mathcal E)$. The first cohomology group is isomorphic to 
$H^2(\mathcal O_Y(-H+F))$ and this group vanishes. 
For the vanishings of the second cohomology group 
it suffices to show the vanishings of
$H^2(\mathcal O_Y(-(a_i+1)F))$ for all $i=1, \dots, n$;  
all of these vanishings occur.
Finally
$H^3(\mathcal E)$ is isomorphic to 
$H^3(\mathcal O_Y(-H-F))$, which also vanishes.
Now we argue for  $Y$ and $B$ as in 
Proposition~\ref{prop.2g.classification} (2c), (2d) or (2e).
To prove the vanishing of 
$H^2(\mathcal T_Y \otimes \mathcal E)$ in these cases we look 
at the exact sequence 
\eqref{eq.tangent.bundle.Hirzebruch.twisted}. Note that
$H^2(\mathcal O_Y(-F))=0$.
On the other hand, by Serre duality 
$h^2(\mathcal O_Y(-2H_0+(1-e)F))=h^0(\mathcal O_Y(-3F))=0$. 
Then $H^2(\mathcal T_Y \otimes \mathcal E)=0$ as wished. 

\smallskip 
\noindent Because  $H^1(\pi^*\mathcal N_{i(Y),\mathbf P^N})$
vanishes, $H^1(\mathcal N_\pi)$ surjects onto $H^1(\mathcal N_\varphi)$. 
If $Y$ and $B$ are as in 
Proposition~\ref{prop.2g.classification} (2a), (2c) or (2d), 
then $H^1(\mathcal N_\varphi)=0$ because $H^1(\mathcal N_\pi)=0$.
If $Y$ and $B$ are as in 
Proposition~\ref{prop.2g.classification} (2e), then 
$h^1(\mathcal N_\varphi)$ is $0$ or $1$ because $h^1(\mathcal N_\pi)=1$.
\end{proof}

\begin{proposition}\label{prop.2g.hom.0}
Let $X$, $Y$, $\mathcal E$, $i$, $\pi$  and $\varphi$ be 
as in Notation~\ref{setup}, \ref{notation.Y} and 
\ref{notation.Y.scroll}. Let $L=\varphi^*\mathcal O_Y(1)$.
If  
$Y$ and $B$ are as in 
Proposition~\ref{prop.2g.classification} (2b), 
then $\varphi$ is induced by the complete linear series 
$|L|$ and 
 $\mathrm{Hom}
  (\mathcal I/\mathcal I^2, \mathcal E)=\mathrm{Ext}^1
  (\Omega_Y, \mathcal E)=0$.
\end{proposition}

\begin{proof}
 Since $B \sim 4H - 4F$, we have $H^0(\mathcal E(1))=
 H^0(\mathcal O_Y(-H+2F))=0$, 
so $\varphi$ is induced by the complete linear series 
$|L|$. 

\smallskip
\noindent
In view of \eqref{eq.cor.Fano}, 
in order to prove the vanishing of $\mathrm{Hom}
  (\mathcal I/\mathcal I^2, \mathcal E)$ it will suffice 
  to prove the vanishing of 
  $H^0(\mathcal T_{\mathbf P^N}|_{i(Y)} \otimes \mathcal E)$ and 
  of $H^1(\mathcal T_Y \otimes \mathcal E)$. 
  Arguing as in the proof of 
  Proposition~\ref{prop.2g.hom.not.0}
it suffices to prove 
$H^0(\mathcal E(1))$ and $H^1(\mathcal E)$ vanish. 
We have already seen $H^0(\mathcal E(1))=0$. For 
$H^1(\mathcal E)=0$, 
since $B \sim 4H-4F$, we have
 and 
 \begin{equation*}
  H^1(\mathcal E)=H^1(\mathcal O_Y(-2H+2F)=0.
 \end{equation*}

 \smallskip
 \noindent Now 
 we prove $H^1(\mathcal T_Y \otimes \mathcal E)=0$. We use 
 exact sequences 
 \eqref{eq1} and  \eqref{eq4}. 
Then it suffices to prove the vanishings of 
$H^1(p^*\mathcal T_{\PP^1} \otimes \mathcal E)$, 
$H^1(p^*E^\vee\otimes \mathcal O_Y(H) \otimes \mathcal E)$ and 
$H^2(\mathcal E)$. 
For the first one, note 
$H^1(p^*\mathcal T_{\PP^1} \otimes \mathcal E)=
H^1(\mathcal O_Y(-2H+4F))=0$.
The vanishing of the second one follows from 
$H^1(\mathcal O_Y(-H+F))=0$. 
The vanishing of the third one 
follows because
\begin{equation*}
  H^2(\mathcal E)=H^2(\mathcal O_Y(-2H+2F)=0.
 \end{equation*}
\end{proof}

\begin{remark}\label{remark.2g.ribbons}
{\rm 
\begin{enumerate}
 \item  It follows from 
 Proposition~\ref{prop.2g.hom.not.0} and 
 \cite[Corollary 1.4]{BE} that there exist nonsplit 
 ribbons supported on $Y$ and with conormal bundle $\mathcal E$, where 
 $Y$ and $\mathcal E$ are as in Proposition~\ref{prop.2g.classification} (2a), 
 (2c), (2d) or (2e). 
 \item It follows from 
 Proposition~\ref{prop.2g.hom.0} and 
 \cite[Corollary 1.4]{BE} that there do not exist nonsplit 
 ribbons supported on $Y$ and with conormal bundle $\mathcal E$, where 
 $Y$ and $\mathcal E$ are as in Proposition~\ref{prop.2g.classification} (2b).
\end{enumerate}
}
\end{remark}

\begin{theorem}
\label{thm.2g.deform.nonhyperelliptic}
 Let $X$ and $Y$ be as in 
notations \ref{setup} (2). 
 Let $(X,L)$ be a hyperelliptic variety such that $L^m=2g$
 and 
 let $\varphi$ be the morphism induced by 
 the complete linear series $|L|$. 
If $Y$ and $B$ are  as in (2a), (2c) or (2d) of 
 Proposition~\ref{prop.2g.classification}, then 
  we have: 
\begin{enumerate}   
 \item The morphism $\varphi$ and the polarized variety $(X,L)$
 are unobstructed.  
\item A general deformation of $\varphi$ is a finite morphism 
of degree $1$ onto 
 its image. 
Likewise, 
a general deformation of $(X,L)$ is nonhyperelliptic
but its complete linear series induces a finite morphism 
of degree $1$ onto 
 its image. 
\end{enumerate}
 \end{theorem}
 
 \begin{proof}
 By Proposition~\ref{prop.2g.hom.not.0} (2),
$H^2(\mathcal O_X)=0$. Thus,  
by Proposition~\ref{prop.algebraic.deformation}, 
there exist
algebraic formally semiuniversal deformations 
 of
$\varphi$ and $(X,L)$. It follows from  
Proposition~\ref{prop.2g.hom.not.0} (4) that $\varphi$ is 
unobstructed. To prove that $(X,L)$ is also 
unobstructed, we see that $H^2(\Sigma_L)=0$. 
For this we use exact sequence \eqref{Atiyah.L}. Since 
$H^2(\mathcal O_X)=0$, we only need to see that
$H^2(\mathcal T_X)$ vanishes. Since 
$H^1(\mathcal N_\varphi)$ vanishes, 
taking cohomology on exact sequence
\eqref{eq.presentation.conormal.bundle}, we 
see that, by the projection formula and the Leray's 
spectral sequence, it is enough to prove the vanishing of 
$H^2(\mathcal T_{\PP^N}|_Y)$ and 
$H^2(\mathcal T_{\PP^N}|_Y \otimes \mathcal E)$. Taking 
cohomology on the restriction to $Y$ of the Euler sequence of
the tangent bundle of $\PP^N$, since  
 $H^2(\mathcal O_Y(1))$, $H^2(\mathcal E(1))$, 
$H^3(\mathcal O_Y)$ and $H^3(\mathcal E)$ vanish, we obtain
the desired vanishings. Then $H^2(\Sigma_L)=0$ and 
by \cite[Theorem 3.3.11 (ii)]{Sernesi}, the polarized variety 
$(X,L)$ is unobstructed. 
This completes the proof of (1).

\smallskip
\noindent 
To prove (2) 
  we are going to use 
  \cite[Theorem 1.4]{deformation}.
    By Proposition~\ref{prop.2g.hom.not.0}, 
 Hom$(\mathcal I/\mathcal I^2,\mathcal E) 
 \neq 0$. If $\mu
\in \mathrm{Hom}(\mathcal I/\mathcal I^2,
\mathcal E)$, $\mu \neq 0$, then $\mu$ is a 
homomorphism of rank $1$ because
$\mathcal E$ is a line bundle.
By Proposition~\ref{prop.2g.hom.not.0},  
 $H^1(\mathcal N_\pi)=0$, so if follows 
 from 
the long exact sequence of cohomology 
\eqref{eq.normal.bundles.cohomology}
that 
 the map $\Psi_2$ 
 of Proposition~\ref{morphism.miguel}
 is surjective.
 Then,  given  a nonzero $\mu
 \in \mathrm{Hom}(\mathcal I/\mathcal I^2,
\mathcal E)$, there exists 
$\nu$ in 
$H^0(\mathcal N_\varphi)$ 
 such that $\Psi_2(\nu)=\mu$. 
Since there is an algebraic formally semiuniversal deformation 
 of $\varphi$ with base $Z$ and $\varphi$ is unobstructed, we see
that all the hypotheses 
of \cite[Theorem 1.4]{deformation} are satisfied, 
so \cite[Theorem 1.4]{deformation} and 
its proof  imply that there exists a smooth, 
algebraic curve $T$ in $Z$, passing through 
 $0$ and tangent to the tangent vector  $v$ of $Z$ 
 corresponding to $\nu$ 
 and a deformation $\Phi_T$ of $\varphi$ 
 over $T$ 
 such that $\Phi_0=\varphi$ and 
 $\Phi_t$ is a finite morphism of degree $1$ onto its image,
 for all $t \in T, 
 t \neq 0$.
Taking the pullback by $\Phi_T$
 of $\mathcal O_{\PP^N_T}(1)$ gives us a 
 deformation $(\mathcal X_T, \mathcal L_T)$
 of $(X,L)$, such that
 for all $t \in T \smallsetminus \{0\}$,
 $(\mathcal X_t, \mathcal L_t)$ is
 nonhyperelliptic. Let $Z'$ be the base of an
 algebraic formally semiuniversal deformation 
 $(\mathcal X, \mathcal L)$
 of
 $(X,L)$. Then, after shrinking $T$ if necessary,
 $(\mathcal X_T, \mathcal L_T)$ is obtained from 
 $(\mathcal X, \mathcal L)$
 by 
 etale base change, so 
 there is a point $z' \in Z'$ such that 
 $(\mathcal X_{z'}, \mathcal L_{z'})$ is nonhyperelliptic 
 but $|\mathcal L_{z'}|$ induces a finite morphism of degree $1$ 
 onto its image. 
Since this is an open condition, this 
finishes the proof of (2).
 \end{proof}

 \begin{remark}
  {\rm 
  The proof of 
  Theorem~\ref{thm.2g.deform.nonhyperelliptic}
  yields a more precise statement for the deformations of 
  $\varphi$. Indeed, note that 
  the element $\nu$ in the proof is a general element of 
  $H^0(\mathcal N_\varphi)$. In fact, $\nu$ belongs to 
  the complement $\mathcal U$ of a linear 
  subespace of $H^0(\mathcal N_\varphi)$, namely, the kernel of 
  $\Psi_2$, which has codimension 
  $h^0(\mathcal N_{i(Y),\PP^N} \otimes \mathcal E)$ 
  in $H^0(\mathcal N_\varphi)$ 
  (for the value of 
  $h^0(\mathcal N_{i(Y),\PP^N} \otimes \mathcal E)$, see
  Proposition~\ref{prop.2g.hom.not.0} (1)). Then, for any 
  $\nu \in \mathcal U$, there exists 
  a smooth, algebraic curve $T$ in $Z$, passing through 
 $0$ and tangent to $\nu$ 
and a deformation $\Phi_T$ of $\varphi$ 
 over $T$ 
 such that $\Phi_0=\varphi$ and, for all $t \in T, t \neq 0$,
 the morphism $\Phi_t$,  we have 
  is finite of degree $1$ 
 onto 
 its image.

 \smallskip
 \noindent
 We can also give a
  similar, more precise statement for the deformations of 
 $(X,L)$. Recall that
 the space $H^1(\Sigma_L)$ parameterizes first--order infinitesimal 
deformations of the pair $(X,L)$ up to isomorphism 
(see  \cite[p. 96]{HarrisMorrison}, 
\cite[Theorem 3.3.11 (ii)]{Sernesi} or 
\cite[pp. 126--128]{Zariski}). 
Taking cohomology in the commutative diagram~\cite[(3.39)]{Sernesi}
one obtains 
the exact sequence 
\begin{equation}\label{eq.cohomology.Ser.3.39}
 H^0(L)^{*} \otimes H^0(L) \longrightarrow 
 H^0(\mathcal N_\varphi) \overset{\zeta}\longrightarrow 
 H^1(\Sigma_L) \longrightarrow 
 H^0(L)^{*} \otimes H^1(L). 
\end{equation}
By the projection formula and the Leray's spectral sequence, 
$H^1(L)=0$, so $\zeta$ is surjective.

\smallskip
\noindent
Consider the commutative diagram 
\begin{equation}\label{eq.comm.square.triangle}
		\xymatrix@C-25pt@R-7pt{
		& H^1(\Sigma_L) \ar[dr]^{\zeta'} 
		\\
	H^0(\mathcal N_\varphi)  \ar[rr]^{\zeta''} \ar[d]^{\Psi_2} 
	\ar[ur]^{\zeta} 
	&& H^1(\mathcal T_X) \ar[d] \\
	H^0(\mathcal N_{i(Y),\PP^N} \otimes \mathcal E)
	\ar[rr]^{\gamma} && H^1(\mathcal T_Y \otimes \mathcal E)}
\end{equation}		
where the square arises from \cite[Proposition 3.7 (1)]{Gonzalez}), 
the homomorphism $\zeta$ is the one in 
\eqref{eq.cohomology.Ser.3.39}, 
the homomorphism $\zeta'$ arises when taking cohomology in 
exact sequence~\eqref{Atiyah.L} and $\zeta$, $\zeta'$ and 
$\zeta''$ are the homomorphisms that forget, in the obvious way,
part of the 
information of the first order infinitesimal deformations 
of $\varphi$ and $(X,L)$.

\smallskip
\noindent
From the commutativity of \eqref{eq.comm.square.triangle} 
and the injectivity of $\gamma$ (in fact, $\gamma$ is 
an isomorphism, see the proof of 
Proposition~\ref{prop.2g.hom.not.0}), it follows that 
$\zeta^{-1}(\zeta(\mathrm{ker}\Psi_2))=\mathrm{ker}\Psi_2$.

\smallskip
\noindent
Then, since there exist elements in $H^0(\mathcal N_\varphi) 
\smallsetminus
\mathrm{ker}\Psi_2$, we have that $\zeta(\mathrm{ker}\Psi_2)$
is a proper vector subspace of $H^1(\Sigma_L)$. Since 
$\zeta$ is surjective, $\mathcal U'=H^1(\Sigma_L) \smallsetminus
\zeta(\mathrm{ker}\Psi_2)$ is a (non empty) open set of 
$H^1(\Sigma_L)$. Then, for any 
  $\varpi \in \mathcal U'$, there exists 
  a smooth, algebraic curve $T'$ in $Z'$, passing through 
 $0$ and tangent to $\varpi$ 
and a deformation $(\mathcal X_{T'},\mathcal L_{T'})$ of 
$(X,L)$ 
 over $T'$ 
 such that $(\mathcal X_{0},\mathcal L_{0})=(X,L)$ and, 
 for all $t' \in T', t' \neq 0$, we have 
  that $|\mathcal L_{t'}|$ induces 
a finite of degree $1$ 
 onto 
 its image.
 
 \smallskip
 \noindent
An analogous remark can be made in relation to 
 Theorem~\ref{thm.Fano-K3.deform.nonhyperelliptic}.}
 \end{remark}

 \begin{theorem}
\label{thm.2g.deform.nonhyperelliptic2}
 Let $X$ and $Y$ be as in 
notations \ref{setup} (2). 
 Let $(X,L)$ be a hyperelliptic variety such that $L^m=2g$
 and 
 let $\varphi$ be the morphism induced by 
 the complete linear series $|L|$. 
If $Y$ and $B$ are  as in (2e) of 
 Proposition~\ref{prop.2g.classification}, 
then a
 general deformation of $\varphi$ is a finite morphism 
of degree $1$ onto 
 its image and  
a general deformation of $(X,L)$ is nonhyperelliptic
but its complete linear series induces a finite morphism 
of degree $1$ onto 
 its image.
 \end{theorem}

\begin{proof}
By Proposition~\ref{prop.2g.hom.not.0} (2),
$H^2(\mathcal O_X)=0$, so, 
by Proposition~\ref{prop.algebraic.deformation}, 
there exist
algebraic formally semiuniversal deformations 
 of
$\varphi$ and $(X,L)$.
We look now at long exact sequence of 
cohomology~\eqref{eq.normal.bundles.cohomology}. 
Recall that in this case $h^1(\mathcal N_\pi)=1$ 
and that we saw in the proof of 
Proposition~\ref{prop.2g.hom.not.0} that $\eta$
is surjective, so that $h^1(\mathcal N_\varphi)$
is $0$ or $1$. Let $Z$ be the base of an 
algebraic formally semiuniversal deformation of $\varphi$. 
We have (see 
\cite[Corollary 2.2.11]{Sernesi})
\begin{equation*}
 h^0(\mathcal N_\varphi) - 
 h^1(\mathcal N_\varphi) \leq 
  \mathrm{dim}Z \leq h^0(\mathcal N_\varphi).
\end{equation*}
We argue first for $h^1(\mathcal N_\varphi)=1$. 
Then $\eta$ is an isomorphism, so 
$\Psi_2$ is surjective.
If 
dim$Z=h^0(\mathcal N_\varphi)$, then 
$\varphi$ is unobstructed, so 
we can argue 
as in
the proof of
Theorem~\ref{thm.2g.deform.nonhyperelliptic} to find 
deformations $\Phi_T$ and $(\mathcal X_T, \mathcal L_T)$ 
over a smooth algebraic curve $T$ in $Z$.

\smallskip
\noindent
Now, If  $h^1(\mathcal N_\varphi)=1$ and 
dim$Z=h^0(\mathcal N_\varphi)-1$, then 
the tangent cone of  $Z$ 
at $0$ has 
codimension $1$ in $H^0(\mathcal N_\varphi)$. 
By Proposition~\ref{prop.2g.hom.not.0}, 
the dimension of 
 Hom$(\mathcal I/\mathcal I^2,\mathcal E) 
 =2$, so ker$\Psi_2$ has codimension $2$ 
 in $H^0(\mathcal N_\varphi)$. Thus 
 the tangent cone of $Z$ is not contained in
ker$\Psi_2$ and there exists an element $\nu$ in the tangent
 cone of $Z$ such that $\Psi_2(\nu)=\mu \neq 0$. 
 Since $\mathcal E$ is a line bundle, 
 $\mu$ is a homomorphism of rank $1$. 
 Since $\nu$ is in the tangent
 cone of $Z$, there exists an algebraic
 curve $\hat T$, $0 \in \hat T$ such that $\nu$ is 
 tangent to $\hat T$. Desingularizing $\hat T$ if 
 necessary
 we obtain a flat family of morphisms 
 satisfying the 
 hypotheses of 
 \cite[Proposition 1.3]{deformation}, so 
there exists a 
 deformation $\Phi_T$ of $\varphi$
 over a smooth algebraic curve $T$ 
 such that the fiber $\Phi_t$ over any 
 $t \in T \smallsetminus \{0\}$ is a morphism  
 to $\PP^N$, which is 
 finite and of degree $1$ onto its image. 
 Now, taking the pullback by $\Phi_T$
 of $\mathcal O_{\PP^N_T}(1)$ gives us a 
 deformation $(\mathcal X_T, \mathcal L_T)$
 of $(X,L)$. Since 
 for all $t \in T \smallsetminus \{0\}$ is of degree $1$,
 $(\mathcal X_t, \mathcal L_t)$ is
 non hyperelliptic. Let $Z'$ the base of an algebraic formally 
 semiuniversal deformation of $(X,L)$.  
 Since $\Phi_T$ and
 $(\mathcal X_T, \mathcal L_T)$ are obtained, by 
 etale base change, from the algebraic formally 
 semiuniversal deformations over, respectively, $Z$ and $Z'$, 
 we may conclude that there are $z \in Z$ and $z' \in Z'$ such that
 $\Phi_z$ and the morphism induced by $|\mathcal L_{z'}|$ are
 finite and of degree $1$ onto its image. 
 
 \smallskip
\noindent 
We argue now for $h^1(\mathcal N_\varphi)=0$.  
In this case $\varphi$ is unobstructed. The 
 kernel of the homomorphism $\epsilon$ of
 \eqref{eq.normal.bundles.cohomology}
 has codimension $1$ 
 in  
 $\mathrm{Hom}(\mathcal I/\mathcal I^2,
\mathcal O_Y) \oplus 
 \mathrm{Hom}(\mathcal I/\mathcal I^2,
\mathcal E)$. By 
Proposition~\ref{prop.2g.hom.not.0},
the linear subspace 
\begin{equation*}
 W=
\mathrm{Hom}(\mathcal I/\mathcal I^2,
\mathcal O_Y) \times \{0\}
\end{equation*}
 has 
codimension $2$ in 
 $\mathrm{Hom}(\mathcal I/\mathcal I^2,
\mathcal O_Y)
 \oplus  \mathrm{Hom}(\mathcal I/\mathcal I^2,
\mathcal E)$. Then the kernel of $\epsilon$ 
is not contained in $W$. Thus there exist
$\nu \in 
H^0(\mathcal N_\varphi)$ such that $\Psi_2(\nu)
\neq 0$. Then 
we can argue 
as in
the proof of
Theorem~\ref{thm.2g.deform.nonhyperelliptic} to find 
deformations $\Phi_T$ and $(\mathcal X_T, \mathcal L_T)$ 
over a smooth algebraic curve $T$ in $Z$.

\smallskip
\noindent In all cases, there are $z \in Z$ 
and $z' \in Z'$ such that
 $\Phi_z$ and the morphism induced by $|\mathcal L_{z'}|$ are
 finite and of degree $1$ onto its image. Thus we may conclude
 the proof of (2) using that 
being a  finite morphism of degree $1$ is an open 
 condition. 
\end{proof}

\begin{question}
 {\rm As seen in the proof of 
 Theorem~\ref{thm.2g.deform.nonhyperelliptic2}, since 
 we do not know if
 $h^1(\mathcal N_\varphi) = 0$ (see Proposition~\ref{prop.2g.hom.not.0}), 
 it is not 
 clear whether $\varphi$ is unobstructed or not. It would be 
 interesting to settle the question one way or the other.}
\end{question}

 \begin{theorem}
  \label{thm.2g.deform.hyperelliptic}
Let $X$ and $Y$ be as in 
notations \ref{setup} (2). 
 Let $(X,L)$ a hyperelliptic 
variety such that $L^m=2g$ and let $\varphi$ 
be the morphism induced by $|L|$. 
If $Y$ and $B$ are as in 
Proposition~\ref{prop.2g.classification} (1) or (2b), 
then $\varphi$ is unobstructed and any deformation of $(X,L)$ is 
 hyperelliptic. 
 \end{theorem}
 
 \begin{proof}
  If $Y=\PP^m$, although the claim about deformations follows 
  from 
  Remark~\ref{remark.deformation.varieties} (3), 
the complete result, including the 
 unobstructedness of $\varphi$, follows
  trivially from Theorem~\ref{Psi2=0.3}.

   \smallskip
  \noindent
  Let now $Y$ and $B$ be 
  as in Proposition~\ref{prop.2g.classification} (2b). 
Let $Z$ be a smooth, algebraic variety with a distinguished point
$0 \in Z$.
Let $(\mathcal X, \mathcal L)$ be a flat family  
over $Z$ such that the fiber $(\mathcal X_0, \mathcal L_0)$ over $0$
is isomorphic to $(X,L)$. 
Since $H^1(\mathcal O_Y(1))=0$ and 
$H^1(\mathcal E(1))=0$,
because $H^1(\mathcal E(1))=H^1(\mathcal O_Y(-H+2F))$, 
we have $H^1(L)=H^1(\pi_*L)=0$, so,
arguing as in the proof of 
Theorem~\ref{thm.Fano-K3.deform.nonhyperelliptic}
and shrinking $Z$ if necessary, we
obtain from $\Phi_*\mathcal L$  a morphism
\begin{equation*}
 \Phi: \mathcal X \longrightarrow \PP^N_Z
\end{equation*}
such that $\Phi_0=\varphi$, i.e., $\Phi$ is a deformation of 
$\varphi$.

\smallskip
\noindent We apply Theorem~\ref{Psi2=0.3} to $\Phi$. 
Proposition~\ref{prop.2g.hom.0} tells 
Hom$(\mathcal I/\mathcal I^2, \mathcal E)=0$, 
so $\Psi_2=0$. As already mentioned, $H^1(\mathcal O_Y)=0$, and 
$H^2(\mathcal O_Y)=0$ also. In addition, 
\begin{equation*}
 H^1(\mathcal E^{-2})=H^1(\mathcal O_Y(B))=H^1(\mathcal O_Y(4H_0))=
 H^1(\mathcal O_{\PP^1})^{\oplus 5}=0 
\end{equation*}
and it is well known that $Y$ is unobstructed in projective space. 
Then all the hypotheses of Theorem~\ref{Psi2=0.3} are satisfied so
$\varphi$ is unobstructed and, 
 for all $z \in Z$, the morphism $\Phi_z$
 has degree $2$ onto 
 its image, which is a deformation of $i(Y)$ in $\PP^N$. 
Since 
 any deformation of $i(Y)$ is variety of minimal degree, 
 the morphism $\Phi_z$ is induced by the complete linear
 series $H^0(\mathcal L_z)$ and $\mathcal L_z^m=L^m$, 
 we conclude that 
 $(\mathcal X_z, \mathcal L_z)$ is a hyperelliptic polarized 
 variety.
 \end{proof}

\section{Deformations of 
generalized hyperelliptic polarized 
varieties.}\label{Section.generalized.hyperelliptic}

\noindent In this section we continue the study of deformations of  certain
generalized hyperelliptic polarized 
varieties, looking this time at Calabi-Yau and general type varieties. If 
$(X,L)$ is either a hyperelliptic polarized Calabi-Yau variety of dimension $m$, 
$m \geq 3$, or a 
hyperelliptic variety of general type, canonically polarized, of dimension $m$, 
$m \geq 2$, then, by adjunction, $L^m < 2g-2$ ($g$ is the sectional genus 
of $(X,L)$)
and
$h^0(\mathcal L_t)$ is constant for any deformation of $(X,L)$
(because $H^1(L)=0$, by the Kodaira vanishing theorem, if $X$ is Calabi-Yau and 
because the invariance by deformation of the geometric genus otherwise); thus by 
Remark~\ref{remark.deformation.varieties} (1), all deformations of 
$(X,L)$ are hyperelliptic. Then in this context it is interesting to 
see  if there are further reasons for this phenomenon, namely, all deformations
of a hyperelliptic polarized variety are hyperelliptic, to happen. That is 
the case, since all the deformations of generalized hyperelliptic Calabi-Yau and 
general type varieties are generalized hyperelliptic,  
as we will see in Theorems~\ref{thm.deformation.CY} and 
\ref{thm.deformation.general.type}. First, we recall the 
definition of Calabi-Yau variety: 

\begin{definition}\label{defi.CY}
 {\rm Let $\mathfrak{X}$ be a smooth variety of dimension $m$, 
 $m \geq 3$. We say that $\mathfrak{X}$ is a Calabi--Yau 
 variety if 
 \begin{enumerate}
  \item $\omega_\mathfrak{X}= \mathcal O_\mathfrak{X}$; and 
  \item $H^i(\mathcal O_\mathfrak{X})=0$, for all $1 \leq i 
  \leq m-1$.
 \end{enumerate}}
\end{definition}

\begin{lemma}\label{lemma.CY}
 {Let $X$, $Y$, $\pi$ and $\mathcal E$ be as in notations \ref{setup} and 
\ref{notation.Y} with $m \geq 3$. If  
$\omega_X=\mathcal O_X$, then $\mathcal E=\omega_Y$
and $X$ is Calabi--Yau.}
\end{lemma}

\begin{proof}
Since $X$ is Calabi-Yau, 
$\pi^*(\omega_Y \otimes \mathcal E^{-1})=\mathcal O_X$, 
 so $\omega_Y \otimes \mathcal E^{-1}$ is numerically trivial. 
 Since 
 $Y$ is either projective space or a hyperquadric or 
 a projective bundle 
 over $\PP^1$, $\omega_Y \otimes \mathcal E^{-1}$ 
 is in fact trivial, so 
 $\mathcal E=\omega_Y$.
 
 \smallskip
 \noindent By the projection formula and the Leray's spectral 
 sequence, $H^i(\mathcal O_X)=H^i(\mathcal O_Y) 
 \oplus H^i(\mathcal E)$. If $1 \leq i \leq m-1$, then 
 $H^i(\mathcal O_Y)=0$ and, by Serre duality
 $H^i(\mathcal E)=H^i(\omega_Y)=H^{m-i}(\mathcal O_Y)=0$. 
 \end{proof}

\begin{proposition}\label{prop.cohomology.vanishing.CY}
Let $X$, $Y$, $\pi$ and $\mathcal E$ be as in notations \ref{setup} and 
\ref{notation.Y} and assume 
$X$ is Calabi-Yau (of dimension $m$, $m \geq 3$).
 Then $H^1(\mathcal T_Y \otimes \mathcal E)=0$. 
\end{proposition}

\begin{proof}
 If $Y=\PP^m$, then $H^1(\mathcal E(1))=H^2(\mathcal E)=0$ by the 
 vanishing of the intermediate cohomology of line bundles in projective space. 
 In view of the Euler sequence for the tangent bundle of $\PP^m$, this 
 implies the vanishing of $H^1(\mathcal T_Y \otimes \mathcal E)$. 
 
  \smallskip
 \noindent If $Y$ is a hyperquadric, then it follows from 
 Lemma~\ref{lemma.CY} that $\mathcal E=\mathcal O_Y(-m\mathfrak h)$. 
 In view of exact sequence for the normal bundle of $Y$ in $\PP^{m+1}$, we need 
 to check the vanishings of $H^0(\mathcal O_Y((-m+2)\mathfrak{h}))$ and 
 $H^1(\mathcal T_{\PP^{m+1}}|_Y \otimes \mathcal E)$. The first happens because 
 $m \geq 3$. The second follows  from 
 the Euler sequence for the tangent bundle of $\PP^{m+1}$ restricted to $Y$, since 
 $H^1(\mathcal E(1))=0$ by the Kodaira vanishing theorem and 
 $h^2(\mathcal E)=h^{m-2}(\mathcal O_Y)=0$.

  \smallskip
 \noindent If $Y$ is a projective bundle over $\PP^1$, we use 
 Notation~\ref{notation.Y.Fano} and exact sequences~\eqref{eq1} and \eqref{eq2}. 
 Then it is enough to prove the vanishings of 
 $H^1(p^*\mathcal T_{\PP^1} \otimes \mathcal E)$, 
 $H^1(\mathcal E \otimes \mathcal O_Y(H_0))$,  
 $H^1(\mathcal E \otimes \mathcal O_Y(H_0 + e_1F)), \dots, 
 H^1(\mathcal E \otimes \mathcal O_Y(H_0 + e_{m+1}F))$ and $H^2(\mathcal E)$. 
 Then $h^1(p^*\mathcal T_{\PP^1} \otimes \mathcal E)=
 h^{m-1}(\mathcal O_Y(-2F))=0$ by the projection formula and the 
 Leray spectral sequence (recall $m \geq 3$). 
 In addition, 
 $h^1(\mathcal E \otimes \mathcal O_Y(H_0))=h^{m-1}(\mathcal O_Y(-H_0))=0$, 
 also by the projection formula and the 
 Leray spectral sequence. The vanishings of 
 $H^1(\mathcal E \otimes \mathcal O_Y(H_0 + e_1F)), \dots, 
 H^1(\mathcal E \otimes \mathcal O_Y(H_0 + e_{m+1}F))$ are argued analogously. 
 Finally, $h^2(\mathcal E)=h^{m-2}(\mathcal O_Y)=0$. 
\end{proof}

\noindent 
In analogy with Definition~\ref{defi.CY}, 
by a Calabi--Yau ribbon we mean a ribbon of dimension bigger than $2$, 
with trivial dualizing sheaf and such that the intermediate cohomology 
of its structure sheaf vanishes. 
We now deduce from Proposition~\ref{prop.cohomology.vanishing.CY} the non 
existence of nonsplit Calabi--Yau ribbons: 

\begin{corollary}\label{cor.CY.ribbons}
Let $Y$ be as in Notations~\ref{setup} (2) and \ref{notation.Y}.
There are no nonsplit Calabi--Yau ribbons  
on $Y$.
\end{corollary}

\begin{proof}
Let $\widetilde Y$ be a ribbon supported on $Y$ and let 
$\widetilde{\mathcal E}$ be the conormal bundle of $Y$ in $\widetilde Y$. 
The same argument of the proof of \cite[Proposition 1.5]{Enriques}
implies that $\widetilde Y$ is Calabi--Yau if and only if 
$\widetilde{\mathcal E}=\omega_Y$.  
Then the result follows from Proposition~\ref{prop.cohomology.vanishing.CY} and 
\cite[Corollary 1.4]{BE}.
\end{proof}

\begin{theorem}\label{thm.deformation.CY}
{Let $X$, $Y$, $\pi$ and  $\varphi$ be as in 
notations \ref{setup} and \ref{notation.Y}. Let 
$X$ be a  Calabi-Yau variety of dimension $m$, $m \geq 3$
and let the morphism $\varphi$  
be  
induced by a complete linear series 
(i.e., $H^0(\mathcal E(1))=0$).
Then 
any deformation of $\varphi$
is a finite morphism of degree
 $2$ onto its image, which is a deformation of $i(Y)$ in $\PP^N$.} 
\end{theorem}

\begin{proof}
Let $Z$ be a smooth algebraic variety
with a distinguished point $0$ and 
let $(\mathcal X, \Phi)$ be a flat family
over $Z$, with $(\mathcal X_0, \Phi_0)=(X, \varphi)$. 
Arguing as in the proof of Proposition~\ref{prop.Fano.deform.lK} 
we show the 
existence of a deformation $(\mathcal X,\Pi)$ of $(X,\pi)$, 
where $\Pi$
 is finite, surjective and of degree $2$, a deformation 
$(\mathcal Y, \mathfrak i)$ of $(Y,i)$ and a 
deformation $\Phi'$ of $\varphi$ such 
$\Phi'=\mathfrak i \circ \pi$. 
We now compare $\Phi$ and $\Phi'$. Let $\mathfrak X$ be 
a general member 
of $|L|$. By adjunction $\mathfrak X$ is of general type, 
$L|_{\mathfrak X}=K_\mathfrak X$ and, since 
$H^1(\mathcal O_X)=0$, the complete linear series 
$|L|$ restricts to the 
complete linear series $|L|_{\mathfrak X}|$. We apply
Lemma~\ref{lemma.deformation.multiples.canonical} to 
$(\mathfrak X, L|_{\mathfrak X})$. Then $\Phi$ and $\Phi'$ 
are the same 
when restricted to $\mathfrak X$, so $\Phi$ is generically of degree $2$ onto 
its image. Since $\pi$ is finite, shrinking $Z$ if necessary, so is 
$\Phi$ onto its image. Therefore, $\Phi_z$ is 
finite of degree $2$ onto 
its image, for all $z \in Z$. 
\end{proof}

\begin{remark}\label{remark.CY}
{\rm Assume $X$, $Y$ and $\varphi$ are as in Theorem~\ref{thm.deformation.CY} and 
that, if $Y$ is a projective bundle over $\PP^1$, then 
$-K_Y$ is base--point--free (i.e., 
$e_1 + \cdots + e_{m-2}
+ (1-m)e_{m-1} \geq -2$, with $e_1, \dots, e_{m-1}$ in  
Notation~\ref{notation.Y.Fano}). Then we can argue as in the proof of 
Theorem~\ref{thm.Fano.deform} and, 
using Theorem~\ref{Psi2=0.3}, conclude that $\varphi$ is unobstructed.}
\end{remark}

\noindent Although we already observed at the beginning of the section 
that 
next result follows from Remark~\ref{remark.deformation.varieties} (1), 
it is 
interesting  to see how it is deduced from the broader setting of 
Theorem~\ref{thm.deformation.CY}: 

\begin{corollary}\label{cor.CY}
 Let $X$ be a  Calabi--Yau variety 
 (of dimension $m$, $m \geq 3$) 
 and let $L$ be a polarization on $X$. 
If $(X,L)$ is hyperelliptic  and 
the image of $X$ by the morphism induced by $|L|$ is smooth, 
then any deformation of $(X,L)$ is hyperelliptic.
\end{corollary}

\begin{example}
  {\rm If $(X,L)$ is a  
polarized Calabi--Yau threefold  
 with $L^3=8$ and $h^0(L)=7$, then, 
 by adjunction and Clifford's theorem, 
 $(X,L)$ is hyperelliptic. 
 Thus, if the image $i(Y)$ of the morphism induced by 
 $|L|$ is smooth, then $i(Y)$ is a rational normal scroll 
 $S(1,1,2)$ of $\PP^6$.
This suggests that 
 polarized Calabi--Yau threefolds $(X,L)$ with $L^3=8$ and $h^0(L)=7$ 
 are parametererized by an irreducible scheme 
 whose general point corresponds to a hyperelliptic 
 polarized Calabi--Yau threefold.} 
 \end{example}

\noindent Corollary~\ref{cor.CY} shows that a 
hyperelliptic polarized Calabi--Yau threefold $(X,L)$ only deforms to 
hyperelliptic polarized Calabi--Yau threefolds. 
If  the image of 
the morphism induced by $|L|$ is a smooth rational normal scroll, 
then $X$
is fibered by 
$K3$ surfaces (see \cite[Proposition 1.6]{CYthreefolds}); thus, in 
this case, 
any deformation of $X$ carries also the $K3$ fibration.
 This motivates the 
following question.

 \begin{question} 
{\rm Let $X$ be a Calabi--Yau threefold.
If $X$ carries a $K3$ fibration, 
does any deformation of $X$ carry the $K3$ fibration?
It is tempting but, probably, too optimistic, to expect that 
such a Calabi--Yau threefold $X$ carries
a $K3$ fibration if and only if there exists a 
polarization $L$ on $X$ so that 
$(X,L)$ is hyperelliptic and 
the image of the morphism induced by $|L|$ is a smooth 
rational normal scroll. 
 If this expectation were true
then one would easily show that 
 the answer to our question is affirmative.}
\end{question}

\noindent Now we study the deformations of generalized hyperelliptic
polarized varieties of general type:

\begin{proposition}\label{prop.general.type}
Let $X$, $Y$, $\pi$ and $\mathcal E$ be as 
in notations \ref{setup} and 
\ref{notation.Y} and 
assume $X$ is a  
 minimal  variety of general type  (i.e., with $K_X$ ample)
 of dimension $m$, $m \geq 2$. 
 Then $H^1(\mathcal T_Y \otimes \mathcal E)=0$. 
\end{proposition}

\begin{proof}
Let $\mathcal E'=\omega_Y \otimes \mathcal E^{-1}$. 
Recall that $\omega_X=\pi^*\mathcal E'$, 
so $\mathcal E'$ is ample. Any ample line bundle on $Y$ is very ample, so $\mathcal E'$ is very ample. 
Then $\mathcal E=\omega_Y \otimes \mathcal E'^{-1}$ and in the proofs of
\cite[Propositions 1.4, 1.5, 1.6]{canonical.double.covers.varieties} (see also 
\cite[Proposition 1.7]{canonical.double.covers.surfaces}) we showed $H^1(\mathcal T_Y \otimes \omega_Y
\otimes \mathcal E'^{-1})=0$. 
\end{proof}

\begin{remark}\label{remark.X.regular} 
 {\rm Under the hypotheses of Proposition~\ref{prop.general.type}, 
 the variety $X$ is regular. Indeed, 
 $\pi_*(\mathcal O_X)=\mathcal O_Y \oplus \mathcal E$. 
 We have $H^1(\mathcal O_Y)=0$. We see that 
 \begin{equation}\label{eq.E.nonspecial}
   H^1(\mathcal E)=0
 \end{equation}
also. 
 If 
 $Y=\PP^m$, then $H^1(\mathcal E)$ vanishes 
 because of the vanishing of intermediate cohomology in 
 $\PP^m$. If $Y$ is a (smooth) hyperquadric and $m \geq 3$, 
 the vanishing of $H^1(\mathcal E)$
 follows from sequence~\eqref{eq.presentation.structure.sheaf.hyperquadric} 
and the vanishing of intermediate cohomology in 
 $\PP^N$. If $Y$ is a projective bundle on $\PP^1$, 
 then $h^1(\mathcal E)=h^{m-1}(\mathcal E')$, with 
 $\mathcal E'$ ample. Then the vanishing follows 
 by the Leray's spectral sequence and the projection formula.}
\end{remark}

\noindent We say that a ribbon of dimension greater than or equal to $2$
is a minimal ribbon  of general type if its dualizing sheaf is ample.  
From Proposition~\ref{prop.general.type} we deduce the non 
existence of nonsplit minimal ribbon  of general type 
supported on $Y$ as in Notation~\ref{notation.Y}: 

\begin{corollary}\label{cor.general.type.ribbons}
Let $Y$ be as in Notation~\ref{notation.Y}.
There are no nonsplit minimal ribbon  of general type 
on $Y$.
\end{corollary}

\begin{proof}
Let $\widetilde Y$ be a minimal ribbon  of general type supported on $Y$ 
and let 
$\widetilde{\mathcal E}$ be the conormal bundle of $Y$ in $\widetilde Y$. 
Since $\omega_{\widetilde Y}$ is ample, by 
\cite[Lemma 1.4]{Enriques}, so is $\omega_Y \otimes {\widetilde{\mathcal E}}^{-1}$ 
and, 
in fact, very ample.  Then 
$\mathrm{Ext}^1(\Omega_Y, \widetilde{\mathcal E})=H^1(\mathcal T_Y \otimes 
\widetilde{\mathcal E})=0$ by 
\cite[Propositions 1.4, 1.5, 1.6]{canonical.double.covers.varieties}
so  the result follows from 
\cite[Corollary 1.4]{BE}.
\end{proof}

\begin{corollary}\label{cor.Hom.general.type}
Let $X$, $Y$, $\varphi$, $\mathcal E$ and $\mathcal I$ be as in 
notations \ref{setup} and \ref{notation.Y}. 
Assume $X$ is a
 minimal  variety of general type  (i.e., with $K_X$ ample)
 of dimension $m$, $m \geq 2$
and $\varphi$ is  induced by a complete linear series
(i.e., $H^0(\mathcal E(1))=0$). 
 Then 
 $\mathrm{Hom}(\mathcal I/\mathcal I^2, \mathcal E)=0$.
\end{corollary}

\begin{proof}
The result follows from exact sequence~\eqref{eq.cor.Fano}, the vanishing of 
$H^0(\mathcal E(1))$ (this is because $\varphi$ 
 factors through $\pi$ and 
 $\varphi$ is induced by a complete linear series), 
 \eqref{eq.E.nonspecial}
 and Proposition~\ref{prop.general.type}. 
 \end{proof}

\begin{theorem}\label{thm.deformation.general.type}
Let $X$, $Y$ and  $\varphi$ be as in 
notations \ref{setup} and \ref{notation.Y}. 
Assume $X$ is a  
 minimal  variety of general type  
 of dimension $m$, $m \geq 2$
and $\varphi$ is  induced by a complete linear series.
If $Y$ is a projective bundle over $\PP^1$, assume 
furthermore $B$ is base--point--free. 
Then $\varphi$ is unobstructed and any deformation of $\varphi$
is a finite morphism of degree
 $2$ onto its image, which is a deformation of $i(Y)$ in $\PP^N$. 
\end{theorem}

\begin{proof}
As seen in the proof of 
Theorem~\ref{thm.Fano.deform}, 
 hypotheses (1), (3) and (4) of 
 Theorem~\ref{Psi2=0.3} are satisfied. 
 If $Y$ is $\PP^m$ or a hyperquadric, then 
 $H^1(\mathcal E^{-2})=0$ because of the vanishing of 
 cohomology in projective space.
 If $Y$ is a projective bundle over $\PP^1$, then 
$B$ being base--point--free implies  \linebreak
$\beta \geq \alpha e_{m-1}$,
 with 
 $\alpha, \beta$ and $e_{m-1}$ as in \eqref{notation.Y.Fano}, so  
 $H^1(\mathcal E^{-2})=0$. Thus hypothesis (2) 
 of Theorem~\ref{Psi2=0.3} is also satisfied.
 Finally, hypothesis (5) of Theorem~\ref{Psi2=0.3}
 follows from Corollary~\ref{cor.Hom.general.type}.
\end{proof}

\begin{remark}
 {\rm The divisor $B$ being base--point--free in the statement of 
 Theorem~\ref{thm.deformation.general.type}
 is not a very restrictive condition. 
Indeed, $\mathcal E^{-2}= \omega_Y^{-2} \otimes \mathcal E'$, 
 where $\mathcal E'$ is ample. If $Y$ is ``balanced'', e.g. 
 if $e_1=\cdots = e_{m-1}=0$, then $\omega_Y^{-2}$ is ample and free, 
 and so is $\mathcal E^{-2}$.}
\end{remark}

\noindent The cases dealt with in the next proposition 
have already been covered by Theorem~\ref{thm.deformation.general.type}
except maybe 
if $B$ is not base--point--free in the statement
of Theorem~\ref{thm.deformation.general.type}; that
is why we include it here.  
The proof is the same as the proof of 
Proposition~\ref{prop.Fano.deform.lK}.

\begin{proposition}\label{prop.general.type.deform.lK}
 Let $X$, $Y$ and  $\varphi$ be as in 
notations \ref{setup} and \ref{notation.Y}. Let 
$X$ be a  minimal  variety of general type
and let the morphism $\varphi$ from $X$ to $\PP^N$ be  
induced by a complete linear series $|L|$, where 
$L=lK_X$ for some $l \in \NN$.
Let $Y$ be a projective bundle over $\PP^1$.
Then any deformation of $\varphi$
is a finite morphism of degree
 $2$ onto its image, which is a deformation of $i(Y)$ in $\PP^N$.
\end{proposition}

\begin{remark}
 {\rm If $X$, $Y$ and  $\varphi$  are as 
 in Theorem~\ref{thm.deformation.general.type}
or Proposition~\ref{prop.general.type.deform.lK} 
and $Y$ is a projective bundle 
over $\PP^1$, but neither is $B$ base--point-free in the statement
of Theorem~\ref{thm.deformation.general.type} nor 
the condition on $L$ in Proposition~\ref{prop.general.type.deform.lK} 
hold, still 
something can be said about the deformations of $\varphi$. 
Indeed, it 
follows from the same reasons argued in 
Remark~\ref{remark.Fano.Wehler} that, 
for any 
deformation $\mathcal X$ of $X$, there exists a 
deformation of $(\mathcal X,\Phi)$ of $(X,\varphi)$
which is finite and of degree $2$ onto its image.}
\end{remark}

\noindent 
{\bf Acknowledgements.}
{We thank Edoardo Sernesi 
 for helpful conversations.}

\end{document}